%% file: main.tex
\documentclass[onefignum,onetabnum]{siamart220329}

\usepackage{braket,amsfonts}
\usepackage{amssymb}
\usepackage{multirow}

\usepackage{array}

\usepackage[caption=false]{subfig}
\usepackage{pgfplots}

\newsiamthm{claim}{Claim}
\newsiamremark{remark}{Remark}
\newsiamremark{hypothesis}{Hypothesis}
\crefname{hypothesis}{Hypothesis}{Hypotheses}

\usepackage{algorithmic}

\usepackage{graphicx,epstopdf}

\Crefname{ALC@unique}{Line}{Lines}

\usepackage{amsopn}

\usepackage{xspace}
\usepackage{bold-extra}
\usepackage[most]{tcolorbox}

\colorlet{texcscolor}{blue!50!black}
\colorlet{texemcolor}{red!70!black}
\colorlet{texpreamble}{red!70!black}
\colorlet{codebackground}{black!25!white!25}

\lstdefinestyle{siamlatex}{%
  style=tcblatex,
  texcsstyle=*\color{texcscolor},
  texcsstyle=[2]\color{texemcolor},
  keywordstyle=[2]\color{texemcolor},
  moretexcs={cref,Cref,maketitle,mathcal,text,headers,email,url},
}

\tcbset{%
  colframe=black!75!white!75,
  coltitle=white,
  colback=codebackground, 
  colbacklower=white, 
  fonttitle=\bfseries,
  arc=0pt,outer arc=0pt,
  top=1pt,bottom=1pt,left=1mm,right=1mm,middle=1mm,boxsep=1mm,
  leftrule=0.3mm,rightrule=0.3mm,toprule=0.3mm,bottomrule=0.3mm,
  listing options={style=siamlatex}
}

\newtcblisting[use counter=example]{example}[2][]{%
  title={Example~\thetcbcounter: #2},#1}

\newtcbinputlisting[use counter=example]{\examplefile}[3][]{%
  title={Example~\thetcbcounter: #2},listing file={#3},#1}

\DeclareTotalTCBox{\code}{ v O{} }
{ 
  fontupper=\ttfamily\color{black},
  nobeforeafter,
  tcbox raise base,
  colback=codebackground,colframe=white,
  top=0pt,bottom=0pt,left=0mm,right=0mm,
  leftrule=0pt,rightrule=0pt,toprule=0mm,bottomrule=0mm,
  boxsep=0.5mm,
  #2}{#1}

\patchcmd\newpage{\vfil}{}{}{}
\begin{tcbverbatimwrite}{tmp_\jobname_header.tex}
\title{Sion's Minimax Theorem in Geodesic Metric Spaces and a Riemannian Extragradient Algorithm
}

\author{Peiyuan Zhang\thanks{Shanghai Qizhi Institute, Shanghai, China (\email{peiyuan.zhang@yale.edu}).}
\and Jingzhao Zhang\thanks{Institute for Interdisciplinary Information Sciences, Tsinghua University, Beijing, China (\email{jingzhaoz@mail.tsinghua.edu.cn}).}
\and Suvrit Sra\thanks{Department of EECS, Massachusetts Institute of Technology, MA, USA (\email{suvrit@mit.edu}).}}

\headers{Minimax in Geodesic Metric Spaces}{Peiyuan Zhang, Jingzhao Zhang, and Suvrit Sra}
\end{tcbverbatimwrite}
\input{tmp_\jobname_header.tex}

\begin{document}
\maketitle

\begin{tcbverbatimwrite}{tmp_\jobname_abstract.tex}
\begin{abstract}
Deciding whether saddle points exist or are approximable for nonconvex-nonconcave problems is usually intractable. This paper takes a step towards understanding a broad class of nonconvex-nonconcave minimax problems that \emph{do remain} tractable. Specifically, it studies minimax problems over geodesic metric spaces, which provide a vast generalization of the usual convex-concave saddle point problems. The first main result of the paper is a geodesic metric space version of Sion's minimax theorem; we believe our proof is novel and broadly accessible as it relies on the \emph{finite intersection property} alone. The second main result is a specialization to geodesically complete Riemannian manifolds: here, we devise and analyze the complexity of first-order methods for smooth minimax problems.
\end{abstract}


\end{tcbverbatimwrite}
\input{tmp_\jobname_abstract.tex}

\input{defs}

\input{sec_1_intro}
\input{sec_2_prelim}
\input{sec_3_minimax}
\input{sec_4_algorithm}
\input{sec_5_applications}
\input{sec_67_discussion}

\bibliographystyle{siamplain}
\bibliography{references}

\end{document}

%% file: defs.tex
\newcommand{\peiyuan}[1]{{\color{blue} Peiyuan's comment: #1}}

\newcommand{\revise}[1]{#1}
\newcommand{\revisenew}[1]{#1}
\newcommand{\jingzhao}[1]{{\color{orange} Jingzhao's comment: #1}}
\newcommand{\suvrit}[1]{{\color{red} Suvrit's comment: #1}}



\crefrangeformat{assumption}{Assumptions~#3#1#4--#5#2#6}
\Crefname{subsection}{Section}{}%

\newcommand{\fact}{\textbf{Fact. }}

\newcommand{\hyphen}{{\text -}}

\renewcommand{\min}{\mathop{\mathrm{min}\vphantom{\mathrm{sup}}}}
\renewcommand{\max}{\mathop{\mathrm{max}\vphantom{\mathrm{sup}}}}
\newcommand{\nfrac}[2]{{#1} / {#2}}


%
\newcommand{\bbR}{{\mathbb{R}}}
\newcommand{\bbE}{{\mathbb{E}}}
\newcommand{\bbN}{{\mathbb{N}}}
\newcommand{\R}{\mathbb{R}}

\newcommand{\bigO}{\mathcal{O}}
\newcommand{\batch}{\mathcal{B}}

\newcommand{\gd}{{\rm gd}}

\newcommand{\CAT}{{\rm CAT}}
\newcommand{\trace}{{\rm tr}}
\newcommand{\diag}{{\rm diag}}
\newcommand{\diam}{{\rm diam}}

\newcommand{\tr}{\text{tr}}
\newcommand{\grad}{\text{grad}}
\newcommand{\hess}{\text{Hess}}
\newcommand{\Exp}{\mathtt{Exp}}
\newcommand{\Log}{\mathtt{Log}}
\newcommand{\asmp}{{\bm A}}

\newcommand{\va}{{\bm a}}
\newcommand{\vb}{{\bm b}}
\newcommand{\vg}{{\bm g}}
\newcommand{\vu}{{\bm u}}
\newcommand{\vv}{{\bm v}}
\newcommand{\vx}{{\bm x}}
\newcommand{\vy}{{\bm y}}
\newcommand{\vz}{{\bm z}}

\newcommand{\vlambda}{\bm \lambda}

\newcommand{\mA}{{\bm A}}
\newcommand{\mB}{{\bm B}}
\newcommand{\mC}{{\bm C}}
\newcommand{\mD}{{\bm D}}
\newcommand{\mG}{{\bm G}}
\newcommand{\mH}{{\bm H}}
\newcommand{\mI}{{\bm I}}
\newcommand{\mU}{{\bm U}}
\newcommand{\mS}{{\bm S}}

\newcommand{\aff}{\text{aff }}
\newcommand{\cl}{\textrm{cl }}
\newcommand{\co}{\textrm{co }}
\newcommand{\dom}{\textrm{dom }}
\newcommand{\mfld}{\mathcal{M}}
\newcommand{\mfldn}{\mathcal{N}}
\newcommand{\spd}{\mathcal{P}}
\newcommand{\sym}{\mathcak{Sym}}
\newcommand{\proj}{\mathcal{P}}
\newcommand{\stf}{\mathcal{V}}
\newcommand{\sph}{\mathcal{S}}
\newcommand{\hyp}{\mathcal{H}}
\newcommand{\epi}{\text{epi }}
\newcommand{\lsc}{\text{lsc }}

\newcommand{\mDelta}{\bm \Delta}

\newcommand{\cT}{{\mathcal T}}
\newcommand{\cM}{{\mathcal M}}
\newcommand{\cN}{{\mathcal N}}
\newcommand{\cX}{{\mathcal X}}
\newcommand{\cU}{{\mathcal U}}
\newcommand{\cV}{{\mathcal V}}
\newcommand{\cY}{{\mathcal Y}}
\newcommand{\cZ}{{\mathcal Z}}

\renewcommand{\d}[1]{\ensuremath{\operatorname{d}\!{#1}}}
\makeatletter
\newenvironment{chapquote}[2][2em]
  {\setlength{\@tempdima}{#1}%
   \def\chapquote@author{#2}%
   \parshape 1 \@tempdima \dimexpr\textwidth-2\@tempdima\relax%
   \itshape}
  {\par\normalfont\hfill--\ \chapquote@author\hspace*{\@tempdima}\par\bigskip}
\makeatother

%% file: sec_1_intro.tex
\section{Introduction}
We study minimax optimization problems of the form
\begin{equation}
  \label{eq:2}
  \min_{x \in X}\max_{y \in Y}\quad f(x,y),
\end{equation}
where the constraint sets $X$ and $Y$ lie in geodesic metric spaces, and $f$ is a suitable bifunction. Problem~\eqref{eq:2} generalizes the standard Euclidean minimax problem where $X \subseteq \bbR^m$ and $Y \subseteq \bbR^n$. 
Minimax problems as such have drawn great attention recently, e.g., in generative adversarial networks~\cite{goodfellow2014generative}, robust learning \cite{ghaoui1997,montanari2015non}, multi-agent reinforcement learning \cite{busoniu4445757}, adversarial training~\cite{goodfellow2014explaining}, among others. 

A common goal of solving minimax problems is to find global saddle points\footnote{Without further qualification, \revise{we refer to global saddle points as saddle points in this paper}.}. A pair $(x^*, y^*)$ is a \emph{saddle point} if $x^*$ is a   minimum of $f(\cdot, y^*)$ and $y^*$ is a   maximum of $f(x^*, \cdot)$. In game theory, a   saddle point is a special Nash equilibrium~\cite{nash1950equilibrium} for a two-player game. When $f$ is convex-concave (i.e., convex in $x$ and concave in $y$), existence of saddle points is guaranteed by Sion's minimax theorem~\cite{sion1958general}, and their computation is often tractable (e.g.,~\cite{nemirovski2004prox}). But without the convex-concave structure,   saddle points may fail to exist, or even when they exist, computing them can be intractable~\cite{daskalakis2009complexity}. Even computing local saddle points with linear constraints is PPAD-complete~\cite{daskalakis2021complexity}. Therefore, \revise{it is natural to pose the following question:}
\vspace{1em}
\begin{center}
  \begin{minipage}{.8\linewidth}
    \it
    \begin{center}
    Which nonconvex-nonconcave minimax problems admit saddle points, and can we compute them?
    \end{center}
  \end{minipage}
\end{center}
\vspace{1em}
While at this level of generality this question is unlikely to admit satisfactory answers, it motivates us to pursue a more nuanced study, and to seek tractable subclasses of problems or alternative optimality criteria---e.g., the works~\cite{jin2020local,mangoubi2021greedy, fiez2021global} explore this topic and establish novel optimality criteria for nonconvex-nonconcave problems. We instead explore a rich subclass of nonconvex-nonconcave problems that do admit   saddle points: minimax problems over \emph{geodesic metric spaces}~\cite{burago2001course}. We provide sufficient conditions that ensure existence of   saddle points by establishing a metric space analog of Sion's theorem. An informal statement of our first main result is as follows:
\begin{theorem}[Informal; see \cref{thm: geodesic-sion}]\label{thm: informal}
    Let $X,Y$ be geodesically convex subsets of 
    geodesic metric spaces $\mfld$ and $\mfldn$, \revisenew{and let $X$ be compact.} If a bifunction $f: X \times Y \to \bbR$ is geodesically (quasi)-convex-concave and (semi)-continuous, \revise{then the equality $\sup_{y\in Y}\min_{x\in X}f(x,y) = \min_{x\in X}\sup_{y\in Y}f(x,y)$ holds.} 
\end{theorem}
\revisenew{If we further assume that both $X$ and $Y$ are compact}, then there exists a saddle point $(x^*, y^*) \in X \times Y$.
\revise{Later in the paper, we will address} computability of saddle points by focusing on the special case of Riemannian manifolds, for which we exploit the available differentiable structure to obtain implementable algorithms. In particular, we devise first-order algorithms for the \emph{Riemannian minimax} problem
\begin{equation} \tag{P} \label{eq: P}
    \min_{x\in \mfld}\max_{y\in \mfldn}\ f(x,y),
\end{equation}
where $\mfld, \mfldn$ are finite-dimensional complete \revise{and connected} Riemannian manifolds, while $f: \mfld \times \mfldn \rightarrow \bbR$ is a smooth geodesically convex-concave bifunction. When the manifolds in~\cref{eq: P} are Euclidean, first-order methods such as optimistic gradient descent-ascent and extragradient (ExtraG) can find saddle points efficiently~\cite{nemirovski2004prox,mokhtari2020unified}. But in the Riemannian case, the extragradient steps do not succeed by merely translating Euclidean concepts into their Riemannian counterparts. We must account for the distortion caused by nonlinear geometry; to that end, we introduce an additional correction that offsets the distortion and thereby helps us obtain a Riemannian corrected extragradient (RCEG) algorithm. In our second main result, we provide non-asymptotic convergence rate guarantees for RCEG, informally stated below. 
\begin{theorem}[Informal, see \cref{thm: convex-concave}]
Under suitable conditions on the finite-dimensional Riemannian manifolds $\mfld$, $\mfldn$, the proposed Riemannian corrected extragradient method admits a curvature-dependent $\bigO(\sqrt{\tau} /\epsilon)$ convergence to an $\epsilon$-approximate   saddle point for geodesically convex-concave problems, where $\tau$ is a constant determined by bounds on curvature of the involved manifolds.
\end{theorem}
Our analysis enables us to efficiently solve minimax problems in nonlinear spaces. We give several examples below.

\subsection{Motivating examples and applications} \label{sec: examples}
Minimax problems on geodesic metric spaces subsume Euclidean minimax problems. The more general structure from nonlinear geometry can offer more concise problem formulations or solutions; and even motivate more efficient algorithms. We mention below several examples of minimax problems on geodesic spaces. Some of the examples possess a geodesically convex-concave structure, whereas others are more general and worthy of  further research. 

\textbf{\emph{Constrained Riemannian optimization.}} The first example is constrained minimization on Riemannian manifolds; see e.g.,~\cite{khuzani2017stochastic,liu2020simple}, which also note applications of constrained Riemannian optimization such as non-negative PCA, weighted MAX-CUT, among others. Here, we tackle the following optimization problem:
\begin{equation} \label{eq: constrained-hadamard}
\begin{split}
    & \min\ g(x), \quad \text{for} \ \revise{\ x \in X \subseteq \mfld}, \\
    \text{s.t.} \quad & h(x) = 0,  \quad h := (h_1, \dots h_n) : \mfld \rightarrow \bbR^n,
\end{split}
\end{equation}
where $\mfld$ is a Riemannian manifold \revisenew{and $X$ is a compact and geodesically convex subset}. The idea is to convert~\cref{eq: constrained-hadamard} into an unconstrained Riemannian minimization problem via the {\it augmented Lagrangian}:
\begin{equation}
  \label{eq:1}
     \max_{\lambda \in \bbR^n} \min_{x \in X \subseteq \mfld} f_\alpha(x, \lambda) := g(x) + \langle h(x), \lambda \rangle - \tfrac{\alpha}{2}\| \lambda \|^2.
\end{equation}
If $g$ and all $\{h_i\}_{i=1}^n$ are continuous and geodesically convex, then~\cref{eq:1} is a geodesically-convex-Euclidean-concave problem. We obtain a strong-duality condition for it as a byproduct of \cref{thm: informal}, leading to the following important corollary:
\begin{corollary}[Informal; see \cref{cor: lagrangian}]
Lagrangian duality holds for geo-desically convex Riemannian minimization problems that have geodesically convex constraints of the form \cref{eq: constrained-hadamard}.
\end{corollary}
Hence, a minimizer of \cref{eq: constrained-hadamard} can be found by solving the saddle point problem~\cref{eq:1}. A detailed statement is in \Cref{sec: empirical}.

\textbf{\emph{Geometry-aware Robust PCA.}} Our second example is on finding principal components of a collection of symmetric positive definite (SPD) matrices. The geometry-aware Principal Component Analysis (PCA) in \cite{pmlr-v45-Horev15} exploits Riemannian structure of SPD matrices that is otherwise disregarded in the Euclidean view. Denote the SPD manifold $\spd(n):=\{ M \in \bbR^{n\times n}: M \succ 0 \text{ and } M = M^\top\}$ and the sphere manifold $\sph(n): = \{ x \in \bbR^{n} : x^\top x = I \}$. Let $\{M_i \in \spd(n) \}_{i=1}^k$ be a set of $k$ observed instances. Then, \revise{robust SPD-PCA} can be stated as:
\begin{equation} \label{eq: robust-pca}
    \max_{M \in \spd(n)} \min_{x \in \sph(n)} f_\alpha(x, M) : = - x^\top M x - \frac{\alpha}{k} \sum_{i=1}^k d_{\spd}(M, M_i),
\end{equation} 
\revise{where $d_{\spd}:\spd(n)\times\spd(n) \to \bbR$ is the Riemannian distance induced by the exponential map on $\spd(n)$} and $\alpha > 0$ controls the penalty. Problem in~\cref{eq: robust-pca} has a \emph{locally} geodesically strongly-convex-strongly-concave structure. We elaborate the property of~\cref{eq: robust-pca} and verify the empirical performance of our proposed algorithm on it in \Cref{sec: empirical}.

\revise{
\textbf{\emph{Robust Riemannian (Karcher) mean.}} A third example is the robust estimation of Karcher mean problem. Given a dataset of SPD matrices $\{M_i\in\spd(n)\}_{i=1}^k$, the Karcher mean is the
unique SPD minimizer of  the sum of squared distance defined as
$$ d(A,B) = \| \log(A^{-1/2} B A^{-1/2})\|_F,$$
where log is the matrix logarithm and $\|\cdot\|_F$ is the Frobenius norm. Despite} \revisenew{being} \revise{a hard problem in Euclidean space, Karcher mean can be efficiently tackled under the Riemannian optimization regime \cite{zhang2016first}. Here, we consider a robust version of Karcher mean problem by introducing auxiliary variables $\{Y_i\in\spd(n)\}_{i=1}^k$: 
\begin{equation} \label{eq: karcher}
    \min_{X \in \spd(n)} \max_{Y_i \in \spd(n)} \sum_{i=1}^kd(X,Y_i) - \alpha \cdot\sum_{i=1}^kd(Y_i,M_i),
\end{equation}
where $\alpha>0$ is the penalty coefficient. With a large enough $\alpha$, the robust Karcher mean is a (globally) geodesically strongly-convex-strongly-concave problem since distance function is geodesically strongly-convex \cite{alimisis2020continuous}.}

\textbf{\emph{Further Robust optimization problems}.} We hope to motivate future study of geodesic minimax problems by also noting some applications without the convex-concave structure; many of these applications arise in robust covariance estimation. 

For instance, \revise{suppose we observe $k$ perturbed points $a_i$ from a manifold subset $\Gamma \subset \mfld$ and aim to estimate their covariance in a robust way, given known mean $\mu$. The objective is then
$$ \min_{a \in \Gamma}\max_{S \in \spd(n)} - \frac{k}{2} \log\det(S) - \frac{1}{2} \Log_a(\mu)^\top S^{-1} \Log_a(\mu) + \frac{\alpha}{k}\sum_{i=1}^nd_{\mfld}(a, a_i) , $$
where $d_{\mfld}$ is the Riemannian distance,} \revisenew{$\alpha > 0$} \revise{is the regularization coefficient and $\Log$ is the} \revisenew{inverse} \revise{exponential map. 
Then by incorporating a robust variable $a$, we instead simultaneously minimize the distance between $a$ and $a_i$'s, and maximize on the SPD manifold to estimate the covariance matrix $S$. The objective is} geodesically concave in $S$~\cite{hosseini2015matrix}, and not necessarily convex in $a$. 
        
Other examples include robust computation of Wasserstein barycenters~\cite{huang2021projection,tiapkin2020stochastic} and computation of operator eigenvalues \cite{pesenson2004approach,riddell1984minimax}. We expect that novel tools for geodesic nonconvex-nonconcave problems will prove valuable for these problems.

\subsection{Related work on non-Euclidean saddle points}
We summarize below related work on the existence of saddle point in nonlinear geometry. Sion~\cite{sion1958general} proved a general minimax theorem for quasi-convex-quasi-concave problems in Euclidean space via the Knaster–Kuratowski–Mazurkiewicz (KKM) theorem and also via Helly's theorem. Nevertheless, Sion's proof relies deeply on linear geometry and can not be directly extended. Several recent works attempt to extend Sion's minimax result to non-Euclidean settings. Notably, in~\cite{KRISTALY2014660,COLAO201261,bento2021elements} the authors establish guarantees on the  existence of Nash equilibria for geodesically convex games on Hadamard manifolds. Our analysis generalizes these results by removing the reliance on Riemannian differential structure along with other additional conditions.

\begin{table}[tpb]
    \centering
    \caption{{\bf Results on  saddle point in non-linear geometry.} We compare our Theorem~\ref{thm: geodesic-sion} with several similar existing results. These results are established for different geometry and relies on different continuity, differentiability and convexity conditions of objective $f$. }
    \label{tbl: techniques}
    \begin{tabular}{c|c|c|c|c}
        \multirow{2}{*}{} & Our result & KKM theory & Fixed point & Nonexp.  \\
        & (Thm.~\ref{thm: geodesic-sion}) & \cite{park2019} & \cite{KRISTALY2014660} & mapping \cite{COLAO201261} \\
        \hline
        Differentiability & Not required & Not required & Subdiff. & Not required \\
        Convexity & Quasi-conv. & Quasi-conv. & Conv. & Conv. \\
        Smoothness & Semi-cont. & Semi-cont. & Cont. subdiff. & Cont. \\
        \hline
        Geometry & Geodesic space & KKM space & Hadamard & Hadamard  \\ 
    \end{tabular}
\end{table}
The closest works to ours are~\cite{park2019,park2010generalizations}, which show that Sion's theorem can be established for the novel KKM space that subsumes Hadamard manifolds. Nevertheless, it remains difficult to verify whether a given geometry satisfies the KKM conditions. In contrast, we generalize Sion's theorem to nonlinear space by providing a new approach based on the finite intersection property in compact spaces. 
Our proof is based on the analysis in~\cite{komiya1988}, which focuses on linear spaces. The original arguments in~\cite{komiya1988} do not critically rely on linear structure; however, their presentation omits many key steps, such as referring to the finite intersection property or providing a step-by-step proof for Lemma~\ref{lemma: main-joint-lower-bound-n}. The missing arguments make it difficult for us to judge whether their analysis holds in nonlinear spaces. We complete the missing parts and confirm that a similar proof can be carried out in geodesic metric space, though we are unable to tell how the author completed the original proof in the first place.  We illustrate the strength of our result by comparing it with existing works in \cref{tbl: techniques}.


%% file: sec_2_prelim.tex
\section{Preliminaries and Notation}
In this section, we introduce our notation by briefly overviewing several definitions in geodesic metric spaces and Riemannian manifolds. For more details, we refer readers to the textbooks \cite{burago2001course, lee2006riemannian, do1992riemannian}.

\subsection{Metric (geodesic) geometry}
A metric space equipped with geodesics is called a geodesic metric space. Examples of geodesic metric space are $\text{CAT}(0)$ spaces or Busemann convex spaces \cite{burago2001course,ivanov2014}. Formally, a {\it metric space} is a pair $(\mfld, d_{\mfld})$ of a non-empty set $\mfld$ and a distance function $d_{\mfld}: \mfld \times \mfld \rightarrow \bbR$ defined on $\mfld$. We occasionally omit the subscript $\mfld$ when it causes no confusion. 

A map $\gamma: [0,1] \subset \bbR \to \mfld$ is called a {\it path} on $\mfld$. For any two points $x,y \in \mfld$, a path $\gamma: [0, 1] \to \mfld$ is referred to as a {\it geodesic} joining $x,y$ if
\begin{align*}
  \text{(1)} \quad \gamma(0) = x, \quad \text{(2)} \quad \gamma(1) = y \quad \text{and} \quad \text{(3)} \quad d_{\mfld}(\gamma(t_1),\gamma(t_2)) = |t_2 - t_1| \cdot d_{\mfld}(x, y),
\end{align*}
for any $[t_1, t_2] \subseteq [0, 1]$. \revise{By definition, a geodesic is continuous (but a path is not necessarily continuous).} 
A metric space $(\mfld, d_{\mfld})$ is called a {\it geodesic metric space} if any two points $x,y \in \mfld$ are joined by a geodesic. Using geodesics, the \emph{concept of convexity} can be established in metric spaces. 

Formally, a non-empty set $X \subset \mfld$ is called a \emph{geodesically convex set}, if every (not necessarily unique) geodesic connecting two points in $X$ lies completely within $X$. Further, we can define the concept of (strongly/quasi-)convex functions.
\begin{definition} [Geodesic (quasi-)convexity] \label{def: g-convex}
  A function $f: \mfld \rightarrow \bbR$ is geodesically convex, if for any $x, y \in \mfld$ and $t \in [0,1]$, \revisenew{for any geodesic $\gamma$ satisfying $\gamma(0) = x$ and $\gamma(1) = y$, the following inequality holds: $f(\gamma(t)) \leq (1-t) f(x) + tf(y)$.} Moreover, we say $f$ is geodesically quasi-convex if $f(\gamma(t)) \leq \max \left\{ f(x), f(y) \right\}$; (concavity and quasi-concavity are defined by considering $-f$).
\end{definition}

\begin{definition} [Geodesic strong convexity] \label{def: g-s-convex}
  A function $f: \mfld \rightarrow \bbR$ is geodesically $\mu$-strongly convex, if for any $x, y \in \mfld$ and $t \in [0,1]$, \revisenew{for any geodesic $\gamma$ satisfying $\gamma(0) = x$ and $\gamma(1) = y$, the following inequality holds: $f(\gamma(t)) \leq (1-t)f(x) + tf(y) - \tfrac{\mu t(1-t)}{2}d_{\mfld}(x,y)$}; (strong concavity is defined by considering $-f$).
\end{definition}

\vspace{0.1cm}
\subsection{Riemannian geometry}
An $n$-dimensional {\it manifold} is a \revise{second countable, Hausdorff} topological space that is {\it locally} Euclidean. A smooth manifold is referred as a {\it Riemannian} manifold if it is endowed with a Riemannian metric $\langle \cdot, \cdot\rangle_x$ on the tangent space $T_x\mfld$, for each $x\in\mfld$. The metric induces a norm on the tangent space, denoted $\| \cdot \|_x$; we usually omit $x$ when it causes no confusion. 

A curve $\gamma: [0,1] \to \mfld$ on Riemannian manifold is a geodesic if it is locally length-minimizing and of constant speed.  
An exponential map at point $x \in \mfld$ defines a mapping from tangent space $T_x\mfld$ to $\mfld$ as $\Exp_x(v) = \gamma(1)$, where $\gamma$ is the geodesic with $\gamma(0) = x$ and $\gamma'(0) = v$. \revise{If geodesic is unique between any two points, we can define} the inverse map as $\Log_x: \mfld \to T_x\mfld$. The exponential map also induces the Riemannian distance as $d_{\mfld}(x,y) = \|\Log_x(y)\|$. A {\it parallel transport} $\Gamma_x^y: T_x\mfld \to T_y\mfld$ provides a way of comparing vectors between different tangent spaces. Parallel transport preserves inner product, i.e., $\langle u, v \rangle_x = \langle \Gamma_x^yu, \Gamma_x^yv\rangle_y$ for points $x,y \in \mfld$ and tangent vectors $u,v\in T_x\mfld$. Unlike Euclidean space, a Riemannian manifold is not always flat. Sectional curvature $\kappa$ (or simply ``curvature'') provides a tool to characterize the distortion of geometry on the Riemannian manifold. 

\revise{To make sure that gradient updates on Riemannian manifolds are well defined, we will restrict our discussion to simply-connected and complete manifolds. A Riemannian manifold is {\it complete} if the exponential map $\Exp_x$ at any point $x \in \mfld$ is defined on the entire tangent space $T_x\mfld$. Assuming finite-dimension, a simply-connected and complete} Riemannian manifold admits at least one geodesic between any two points (Hopf-Rinow theorem \cite{lee2006riemannian}). Hence, it inherits the definition of geodesically convex sets and geodesically convex/concave functions in geodesic space.  \revise{In particular, a Hadamard manifold is a special case of such a manifold with non-positive curvature and therefore has unique geodesic between any two points \cite{lee2006riemannian}}. We can readily verify that the inherited convexity is consistent with the usual definition of geodesic convexity in Riemannian optimization literature. 
\begin{lemma} \label{lemma: convex-equivalence}
  A differentiable function $f: \mfld \rightarrow \bbR$ is geodesically convex \revise{if and only if} for any two points $x,y \in \mfld$, 
  \begin{align*}
    f(y) \geq f(x) + \langle \nabla f(x), \Log_x(y) \rangle.
  \end{align*}
  Besides, $f$ is geodesically $\mu$-strongly convex \revise{if and only if} for any two points $x,y \in \mfld$,
  \begin{align*}
    f(y) \geq f(x) + \langle \nabla f(x), \Log_x(y) \rangle + \tfrac{\mu}{2} d_{\mfld}^2(x, y).
  \end{align*}
\end{lemma}
We also state the Lipschitz regularity of smooth functions on Riemannian manifolds using the aforementioned manifold operations.
\begin{definition} \label{def: g-smooth}
  $f$ is geodesically Lipschitz smooth \revisenew{with} modulus $L$, if for any \revisenew{$x, y \in \mfld$, it holds that $\|\nabla f(x) - \Gamma_y^x \nabla f(y)\| \leq L d_{\mfld}(x, y).$}
\end{definition}


%% file: sec_3_minimax.tex
\vspace{0.2cm}

\section{Main theorem: minimax in nonlinear geometry}\label{sec:main}
In Euclidean space, Sion's minimax theorem guarantees strong duality for suitable convex-concave minimax problems. In this section, we establish an analog of Sion's theorem in geodesic metric spaces. The result automatically applies to \revise{complete and connected} Riemannian manifolds as they are just instances of geodesic metric spaces.

We consider the general form of \cref{eq: P} in geodesic metric spaces, i.e., $\mfld$, $\mfldn$ are geodesic metric spaces, $f|_{X\times Y}$ is a geodesically (quasi-)convex-concave bifunction restricted to \revise{compact convex subset $X\subseteq \mfld$ and convex subset $Y \subseteq \mfldn$}. We present below our main theorem that guarantees the existence of a saddle point for this general minimax problem.
\begin{theorem} [Sion's theorem in geodesic metric space] \label{thm: geodesic-sion}
Let $(\mfld,d_\mfld)$ and $(\mfldn,d_\mfldn)$ be geodesic metric spaces. Suppose $X \subseteq \mfld$ is a compact and geodesically convex set, and $Y \subseteq \mfldn$ is a geodesically convex set. If the following conditions hold for the bifunction $f: X \times Y \to \bbR$:

    (1) $f(\cdot, y)$ is geodesically-quasi-convex and lower semi-continuous; and

    (2) $f(x, \cdot)$ is geodesically-quasi-concave and upper semi-continuous.\\
    Then, we have the equality
    \begin{align*}
        \min_{x\in X}\sup_{y\in Y} f(x,y) = \sup_{y\in Y}\min_{x\in X} f(x,y).
    \end{align*}
\end{theorem}

\begin{remark}
In \cref{thm: geodesic-sion}, to keep the statement minimal, we only require $X$ to be compact. Hence, due to the absence of compactness, $f(x, \cdot)$ only admits a supremum but not necessarily a maximum on $Y$ for any $x \in X$. 
\end{remark}

\subsection{Proof of \cref{thm: geodesic-sion}} \label{sec: proof-sion}
We now prove \cref{thm: geodesic-sion}, while postponing proofs of the technical lemmas to \Cref{sec: key-lemma2} and \Cref{sec: key-lemma}.

Note that we restrict the domain to geodesically convex sets $X$ and $Y$ on metric spaces $(\mfld, d_\mfld)$ and $(\mfldn,d_\mfldn)$, respectively. 
Hence it follows from the max-min inequality that  $$\sup_{y\in Y} \min_{x \in X} f(x,y) \leq \min_{x \in X} \sup_{y \in Y} f(x,y).$$

We now prove its reverse. The technique we use generalizes~\cite{komiya1988}. We notice that the function $g(x) = \sup_{y\in Y} f(x,y)$ is lower semi-continuous due to the fact that the supremum of any collection of lower semi-continuous functions is still lower semi-continuous. Combined with compactness of $X$, we deduce by the Weierstrass minimum theorem that $\min_{x \in X} \sup_{y\in Y} f(x,y)$ is bounded away from $-\infty$. Therefore, there exists at least one $\alpha > -\infty$ such that $\alpha < \min_{x \in X} g(x) = \min_{x \in X} \sup_{y\in Y} f(x,y)$.

Now, the major difficulty is to ensure that for any value $\alpha < \min_{x} \sup_{y}f(x,y)$, there is always a point $y_0 \in Y$ such that the condition $\alpha < \min_{x} f(x,y_0)$ holds. To this end, we specify the following claim.
\begin{claim} \label{cl: claim1}
    For any value $\alpha < \min_{x} \sup_{y}f(x,y)$, there exist (finite) $k$ points $y_1, \dots, y_k$ in $Y$ such that condition
    $
        \alpha < \min_{x} \max_{i\in [k]}f(x,y_i) 
    $
    holds.
\end{claim}

The claim follows by connecting the statement with the finite intersection property via geodesic quasi-convexity. Its complete proof can be found in the next \Cref{sec: key-lemma2}. 

\revise{
In the light of \cref{cl: claim1}, we can invoke \cref{lemma: main-joint-lower-bound-n}} below and show that there exists at least one point $y_0 \in Y$ such that 
\begin{equation}\label{eq: exist-y_0}\alpha < \min_{x\in X} f(x, y_0) \leq \sup_{y\in Y}\min_{x\in X} f(x, y).
\end{equation}

\begin{lemma} \label{lemma: main-joint-lower-bound-n}
Under the conditions of \cref{thm: geodesic-sion}, for any finite set of $k$ points $y_1, \dots, y_k$ in $Y$ and any real number $\alpha < \min_{x \in X} \max_{i \in [k]} f(x, y_i)$, there exists a point $y_0 \in Y$ such that $\alpha < \min_{x\in X} f(x, y_0)$.
\end{lemma}

Since the above inequality \cref{eq: exist-y_0} holds for arbitrary $\alpha < \min_{x} \sup_{y}f(x,y)$, by considering a monotonically increasing sequence $\alpha_k \to \min_{x \in X} \sup_{y\in Y} f(x,y)$, we know that 
$$\min_{x \in X} \sup_{y\in Y} f(x,y) = \lim_k \alpha_k  \leq \sup_{y\in Y}\min_{x\in X} f(x, y),$$
which completes the proof of the theorem.

\subsection{Proof of \cref{cl: claim1}} \label{sec: key-lemma2}
\revise{To prove \cref{cl: claim1}, we invoke the {\it finite intersection property} to find a finite number of points fulfilling the statement of the claim. Before proceeding to the proof, we first present the definition and a proposition on the finite intersection property.
}
\begin{definition}[Finite intersection property]
\revise{For a set $C$ and an index set $I$, the collection of subsets $C_i \subset C$, $i \in I$ admits the finite intersection property if for any finite subcollection $C_j, j \in J \subset I$, it holds that $\bigcap_{j\in J}C_{j} \neq  \varnothing$.}
\end{definition}

\begin{proposition}[Theorem 26.9 in \cite{munkres1974topology}] \label{prop: fip}
\revise{Let $C$ be a topological space. Then $C$ is compact if and only if for any collection of closed subsets $C_i \subset C, i \in I$ that admits the finite intersection property, it holds that  $\bigcap_{i\in I}C_{i} \neq \varnothing$.}
\end{proposition} 

\revisenew{
We now define the \emph{level set} of function $f: X \times Y \to \bbR$ with respect to the first variable as 
\begin{align*}
    \phi_y(\alpha) := \{ x \in X | f(x,y) \leq \alpha \}.
\end{align*}
Analogous to the Euclidean case, $\phi_{y}(\alpha)$ is a geodesically convex set if $f(\cdot, y)$ is a geodesically (quasi)-convex function, and it is closed if $f(\cdot, y)$ is lower semi-continuous. For any value $\alpha>-\infty$, the inequality $\alpha < \min_{x}\sup_{y} f(x,y)$ \revisenew{is equivalent to say that $$\cap_{y \in Y} \phi_y(\alpha) = \varnothing.$$ Suppose the latter does not hold, then for such  $\alpha$ there exists at least one $x_0$ in $\cap_{y \in Y} \phi_y(\alpha)$. By the intersection of level sets, we have $f(x_0,y)\leq \alpha$ for any $y\in Y$, and thus $\sup_{y\in Y}f(x_0,y)\leq \alpha$. But this conclusion contradicts the condition that $\alpha < \min_{x}\sup_{y} f(x,y)$. Every step is reversible so the equivalence holds.}

We further notice that for each $y \in Y$, the level set $\phi_y(\alpha)$ is closed and geodesically convex due to lower semi-continuity and quasi-convexity of $f(\cdot, y)$.
Together, we have (1) $X$ is compact, (2) $\phi_y(\alpha)$ is closed, and (3) $\cap_{y \in Y} \phi_y(\alpha) = \varnothing$. By \cref{prop: fip}, the collection of subsets $\phi_y(\alpha)$ of $X$ does not admit the finite intersection property. Therefore, by definition, there exists a finite subset of $k$ points $\{y_1,\dots,y_k\} \subset Y$ such that $\cap_{i \in [k]} \phi_\alpha(y_i) = \varnothing$. So \cref{cl: claim1} is true.}

\subsection{Proof of \cref{lemma: main-joint-lower-bound-n}} \label{sec: key-lemma}
As stated in the previous section, the only missing piece in the proof of \cref{thm: geodesic-sion} is \cref{lemma: main-joint-lower-bound-n}, which serves as an extension of \Cref{lemma: joint-lower-bound} below. This lemma in turn is inspired by and can be regarded as the geodesic version of Lemma 1 in \cite{komiya1988}.

\begin{lemma} \label{lemma: joint-lower-bound}
Under the conditions of \cref{thm: geodesic-sion}, for any two points $y_1, y_2 \in Y$ and any real number $\alpha < \min_{x \in X} \max \{f(x, y_1), f(x, y_2)\}$, there exists a point $y_0 \in Y$ such that $\alpha < \min_{x\in X} f(x, y_0)$.
\end{lemma}
\begin{proof}
The proof is by contradiction. Assume  therefore, that for such an $\alpha$ the inequality $\min_{x\in X} f(x, y) \leq \alpha$ holds for arbitrary $y \in Y$. As a consequence, there exists a constant $\beta$ such that
\begin{align}
\sup_{y\in Y}\min_{x\in X} f(x, y) \leq \alpha < \beta < \min_{x \in X}\,  \max \{f(x, y_1), f(x, y_2)\}.  \label{eq: joint-lower-bound-metric-1} 
\end{align}

\revise{
Consider now a geodesic $\gamma_y: [0, 1] \rightarrow Y$ (recall $Y$ is geodesically convex) connecting $y_1$ and $y_2$. For any $t \in [0, 1]$ and corresponding $z=\gamma_y(t)$ on the geodesic, the level sets $\phi_z(\alpha)$ and $\phi_z(\beta)$ are nonempty due to \cref{eq: joint-lower-bound-metric-1}, and closed due to lower semi-continuity of $f$ in the first variable. And since $f$ is geodesically quasi-concave in the second variable, we obtain the inequality 
\begin{align*}
    f(x, z) \geq \min\{ f(x,y_1), f(x,y_2)\}, \quad \text{with  } z=\gamma_y(t), \quad \forall x\in X,  \quad \forall t \in [0,1].  
\end{align*}
This bound is equivalent to saying that $\phi_{z}(\alpha) \subseteq \phi_{z}(\beta) \subseteq \phi_{y_1}(\beta) \cup \phi_{y_2}(\beta)$.

We then argue that the intersection $\phi_{y_1}(\beta) \cap \phi_{y_2}(\beta)$ should be empty. Otherwise, there exists a point $x\in X$ such that $\max\{f(x,y_1), f(x,y_2)\} \leq \beta$, contradicting \cref{eq: joint-lower-bound-metric-1}. Next, by quasi-convexity, since the level set $\phi_z(\beta)$ is geodesically convex for any $z$, it is also connected. Consider now the three facts:
\begin{itemize}
    \item $\phi_{z}(\alpha) \subseteq \phi_{z}(\beta) \subseteq \phi_{y_1}(\beta) \cup \phi_{y_2}(\beta)$;
    \item $\phi_{y_1}(\beta) \cap \phi_{y_2}(\beta)$ is empty;
    \item $\phi_{z}(\alpha)$, $\phi_{z}(\beta)$, $\phi_{y_1}(\beta)$ and $\phi_{y_2}(\beta)$ are closed, connected, and convex.
\end{itemize}

We then claim that for any point $z$ on the geodesic $\gamma$, either the inclusion $\phi_{z}(\beta) \subseteq \phi_{y_1}(\beta)$ or the inclusion $\phi_{z}(\beta) \subseteq \phi_{y_2}(\beta)$ holds. Suppose not, then we can find two points $x_1 \in \phi_{y_1}(\beta)$ and $x_2 \in \phi_{y_2}(\beta)$ such that both $x_1, x_2 \in \phi_{z}(\beta)$. But since $\phi_{z}(\beta)$ is convex, there is a geodesic $\gamma_x:[0,1] \to X$ in $\phi_{z}(\beta)$ connecting $x_1, x_2$. Therefore, $\gamma_x$ also lies in $\phi_{z}(\beta) \subseteq \phi_{y_1}(\beta) \cup \phi_{y_2}(\beta)$. Because $\phi_{y_1}(\beta) \cap \phi_{y_2}(\beta)$ is empty, the map $\gamma_x^{-1}$ induces a partition $\{J_1,J_2\}$ of $[0,1]$ into $J_1 \cap J_2 = \varnothing$ and $J_1 \cup J_2 = [0,1]$, where $\gamma_x(J_1) \subseteq \phi_{y_1}(\beta)$ and $\gamma_x(J_2) \subseteq \phi_{y_2}(\beta)$. \revisenew{Since $\gamma_x^{-1}$ is continuous, and set $\phi_{y_1}(\beta),\phi_{y_2}(\beta)$ are closed, we can conclude that both $J_1$ and $J_2$ are also closed considering the fact that $J_i=\gamma_x^{-1}(\phi_{y_i}(\beta))$ for $i =\{1,2\}$. This then contradicts the connectedness of $[0,1]$.} 

Because for any $t \in [0,1]$, either $\phi_{\gamma_y(t)}(\alpha) \subseteq \phi_{\gamma_y(t)}(\beta) \subseteq \phi_{y_1}(\beta)$ or $\phi_{\gamma_y(t)}(\alpha) \subseteq \phi_{\gamma_y(t)}(\beta) \subseteq \phi_{y_2}(\beta)$ holds, we know that $\gamma_y^{-1}$ also induces a partition of $[0,1]$ into $I_1$ and $I_2$ defined as follows,
\begin{align*}
    I_1 := \{t \in [0,1] |\ \phi_{\gamma_y(t)}(\alpha) \subseteq \phi_{y_1}(\beta) \}, \\
    I_2 := \{t \in [0,1] |\  \phi_{\gamma_y(t)}(\alpha) \subseteq \phi_{y_2}(\beta) \}.
\end{align*}

We conclude the proof of this lemma by showing that there is a contradiction to the continuity of $\gamma_y$,  the connectedness of interval $[0,1]$, and the upper semi-continuity of $f$ in $y$. The reasoning is as follows: let $t_k$ be an infinite sequence in $I_1$ with limit $t_{\infty}$, we want to show $t_{\infty}$ is also in $I_1$. This claim is done as follows. Consider any $x \in \phi_{\gamma(t_{\infty})}(\alpha)$; upper semi-continuity of $f(x, \cdot)$ implies that
\begin{equation*}
    \underset{k\to \infty}{\lim\sup}f(x,\gamma(t_{k})) \leq\ f(x, \gamma(\revisenew{t_{\infty}})) \leq \alpha \revisenew{\ < \beta}.
\end{equation*}
Therefore, there exists  a large enough integer $l$ such that $f(x,\gamma(t_{l})) < \beta$. This inequality implies that $x \in \phi_{\gamma(t_l)}(\beta)$. We further know from $t_l \in I_1$ that $\phi_{\gamma(t_l)}(\beta) \cap \phi_{y_1}(\beta) \supseteq \phi_{\gamma(t_l)}(\alpha) \neq \varnothing$. Then, upon noting that $ \phi_{y_1}(\beta) \cap \phi_{y_2}(\beta) = \varnothing$ and $\phi_{\gamma(t_l)}(\beta) \subseteq \phi_{y_1}(\beta) \cup \phi_{y_2}(\beta),$ we  conclude that $x \in \phi_{\gamma(t_l)}(\beta) \subseteq \phi_{y_1}(\beta)$ \revisenew{with a analogous argument showing $\phi_z(\alpha) \subseteq \phi_{y_1}(\alpha) \cup \phi_{y_2}(\alpha)$ above.}

Hence, for any $x \in \phi_{\gamma(t_{\infty})}(\alpha)$, the condition $x \in \phi_{y_1}(\beta)$ also holds. In other words, the inclusion $\phi_{\gamma(\revisenew{t_{\infty}})}(\alpha) \subseteq \phi_{y_1}(\beta)$ holds. Thus, by the definition of $I_1$, we know that the limit point $t_{\infty}$ lies in $I_1$, and thus $I_1$ is closed. \revisenew{By a similar argument, we can show that $I_2$ is also closed. Since both $I_1$ and $I_2$ are closed, this is in contradiction with the continuity of $\gamma_y$ and the connectedness of $[0,1]$. Thus we prove the lemma.
}}
\end{proof}

\cref{lemma: main-joint-lower-bound-n} extends the conclusion of \cref{lemma: joint-lower-bound} to any finite $k$ points, and then provides a basis for using the finite intersection property and \cref{cl: claim1}. As argued in \Cref{sec: proof-sion}, this step is key in the proof of \cref{thm: geodesic-sion}. We now state the proof of \Cref{lemma: main-joint-lower-bound-n}.
\begin{proof}[Proof of \cref{lemma: main-joint-lower-bound-n}]
The proof is by induction on \cref{lemma: joint-lower-bound}. For $k=1$, the result is trivial. We assume the lemma holds for $k-1$. \revise{Now, for any $k$ points $y_1, \dots, y_k \in \mfldn$, and any value $\alpha < \min_{x\in X}\max_{i\in[k]}f(x,y_i)$, we denote the set $X’ = \phi_{y_k}(\alpha) $. If $X'$ is empty, selecting $y_0 = y_k$ yields the conclusion. Otherwise, we have
\begin{align*}
    \alpha < \min_{x \in X} \max_{i \in [k]} f(x, y_i) \leq \min_{x \in X'} \max_{i \in [k]} f(x, y_i) \leq \min_{x \in X'} \max_{i \in [k-1]} f(x, y_i),
\end{align*}
where the second inequality is due to $X' \subseteq X$, while the third inequality is due to the fact $f(x,y_k) \le \alpha$ for any $x 
\in X'$.  Due to the definition of level sets  and since $X'\subseteq X$, the set $X'$ is  geodesically convex and compact. We apply our assumption on the $k-1$ points $y_1,\dots,y_{k-1}$ and on the sets $X'$, $Y$, to claim that} there exists a point $y'_0$ such that $\alpha < \min_{x \in X'} f(x, y'_0)$. As a result, we have $\alpha < \min_{x \in X} \max \{f(x, y'_0), f(x, y_k)\}$. Then applying \cref{lemma: joint-lower-bound} leads to the conclusion.
\end{proof}

\subsection{Existence of saddle point in Riemannian minimax problems}
Later in this paper, we \revise{specialize to} the Riemannian minimax problem~\cref{eq: P}. We state Sion's theorem on Riemannian manifolds as a corollary. To guarantee the existence of a pair \revise{of points comprising a} saddle point, we further require set $Y$ to be compact.

\begin{corollary} \label{cor: riemannian-sion}
Suppose that $\mfld$ and $\mfldn$ are finite-dimensional \revise{complete and connected} Riemannian (sub)-manifolds. If subsets $X$, $Y$ and the bifunction $f$ satisfy the condition in \cref{thm: geodesic-sion}, and additionally, $Y$ is also compact, then the following min-max identity holds:
\begin{align*}
    \min_{x\in X}\max_{y\in Y} f(x,y) = \max_{y\in Y}\min_{x\in X} f(x,y).
\end{align*}
\end{corollary}
\begin{proof}
Immediate from \cref{thm: geodesic-sion} as $\mfld$ and $\mfldn$ are geodesic metric spaces. 
\end{proof}
By~\cref{cor: riemannian-sion} we deduce that there is at least one saddle point $(x^*, y^*)$ such that:
\begin{equation*}
    \min_{x\in X} f(x, y^*) = f(x^*, y^*) = \max_{y\in Y} f(x^*, y).
\end{equation*}
If $f$ is geodesically convex-concave, the minimax problem \cref{eq: P} can be tackled by \revise{closing the {\it duality gap}, defined for a given pair $(\hat{x}, \hat{y})$ as
\begin{equation*}
  \text{gap}_f(\hat{x},\hat{y}) := \max_{y} f(\hat{x}, y) - \min_{x} f(x, \hat{y}).
\end{equation*}
The duality gap then serves as an optimality criterion as in the Euclidean setup.}
\begin{definition}
The pair $(\hat{x}, \hat{y})$ is an $\epsilon$-saddle point of $f$, if $\mathrm{gap}_f(\hat{x},\hat{y}) \le \epsilon$. 
\end{definition}
We use this definition when stating non-asymptotic convergence bounds for our Riemannian minimax optimization algorithm.


%% file: sec_4_algorithm.tex
\section{Riemannian Minimax Algorithms and Analysis} \label{sec: algorithm}
In this section we present our algorithm for minimax optimization of a geodesically convex-concave bifunction $f$ on Riemannian manifolds under a suitable smoothness assumption. Building upon the aforementioned optimality criterion, we establish convergence rate of our algorithm via a non-asymptotic analysis. This result is summarized in \cref{tbl: complexity}.

\begin{table}[ht]
    \centering
    \def\arraystretch{1.1}
    \caption{{\bf Comparison of minimax algorithms.} The table summarizes the convergence properties of our RCEG and presents a comparison with the Euclidean counterparts. SC-SC denotes the strongly-convex-strongly-concave case. We provide an explanation of each symbol. $L$: Lipschitz constant of $f$. $\mu$: strong-convexity/concavity constant of $f$. $\tau$: a constant parameterized by curvature and domain diameter (see below and Theorem~\ref{thm: convex-concave}). }
    \begin{tabular}{c|c|c|c|c}

        Geometry &
        Setting &
        Algorithm & 
        Complexity & 
        Reference \\
     
        \hline 
        
        Riemannian &
        convex-concave & 
        RCEG & 
        $\bigO\left(\nfrac{\sqrt{\tau}}{\epsilon}\right)$ & 
        Theorem~\ref{thm: convex-concave} \\

        
        Euclidean &
        convex-concave & 
        ExtraG & 
        $\bigO\left(\nfrac{1}{\epsilon}\right)$ &
        \cite{nemirovski2004prox} \\
        
        Euclidean &
        SC-SC &
        ExtraG & 
        $\bigO\left(\frac{L}{\mu}\log(\nfrac{1}{\epsilon})\right)$ &
        \cite{mokhtari2020unified} \\
    \end{tabular}
    \vspace{0cm}
    \label{tbl: complexity}
\end{table}

Specifically, we consider smooth minimax optimization of \cref{eq: P}. 
To this end, we assume the following regularity conditions.
\begin{assumption} \label{asmp: smooth}
    The gradients of $f$ are geodesically $L$-smooth, i.e., for any two pairs $(x, y)$ and $(x', y') \in \mfld \times \mfldn$, the gradient satisfies the bounds
    \begin{align*}
        & \| \nabla_x f(x, y) - \Gamma_{x'}^{x} \nabla_x f(x', y')\| \leq L \left(d_\mfld(x, x') + d_\mfldn(y, y')\right), \\
        & \| \nabla_y f(x, y) - \Gamma_{y'}^{y} \nabla_y f(x', y')\| \leq L \left(d_\mfld(x, x') + d_\mfldn(y, y')\right).
    \end{align*}
\end{assumption}

\begin{assumption} \label{asmp: convex-concave}
   The bifunction $f|_{X\times Y}$ is geodesically convex in the first variable and geodesically concave in the second variable.
\end{assumption}
 \revise{The next assumption makes sure that any two points on the manifold can be connected by a geodesic.
\begin{assumption}
    Both $\mfld$ and $\mfldn$ are simply-connected and complete Riemannian manifolds.
\end{assumption}

Further, we require the curvature of $\mfld$ and $\mfldn$ to be bounded in range $[\kappa_{\min}, \kappa_{\max}]$. An additional bound on the diameter is necessary when positive curvature is involved, i.e., $\kappa_{\max} > 0$. It allows us to (1) use {\it comparison inequalities} (see \cref{lemma: tri-ineq-lower}, \cref{lemma: tri-ineq-nonpos}), and (2) to ensure that the geodesic is unique between any two points \cite{lee2006riemannian}, so that we can use the \emph{inverse exponential map} $\Log$. We emphasize the assumption is purely algorithmic and independent from our geodesic Sion theorem. This is a regularity condition in Riemannian optimization literature \cite{alimisis2020continuous,zhang2016first}. Formally, it is stated in the following assumption.}
\begin{assumption} \label{asmp: sec-curvature}
    The sectional curvatures of $\mfld, \mfldn$ lie in the range $[\kappa_{\min}, \kappa_{\max}]$ with $\kappa_{\min} \leq 0$. Moreover, if $\kappa_{\max} > 0$, the diameter of the corresponding manifold is (strictly) upper bounded by $\pi / \sqrt{\kappa_{\max}}$.
\end{assumption}

\subsection{Comparison inequalities}
The convergence rates of gradient methods on Riemannian manifolds are \revise{often} curvature dependent. Hence, before we present our convergence analysis, we summarize how the bound on curvature leads to \emph{trigonometric comparison inequalities}. \revise{Suppose there is a {\it geodesic triangle} $\Delta pqr$ with vertices $p,q,r \in X \subset \mfld$ and geodesic edges $\gamma_{pq},\gamma_{qr}, \gamma_{rp}$. Comparison inequalities provide a quantitative relationship between the lengths of geodesic edges.} A first result is due to \cite{zhang2016first}, which is obtained when the sectional curvature is bound from below.
\begin{lemma} [Lemma 5 in \cite{zhang2016first}] \label{lemma: tri-ineq-lower}
Let $\mfld$ be a Riemannian manifold with sectional curvature lower bounded by $\kappa_{\min} \leq 0$. \revise{If $a,b,c$ are the length of sides $\gamma_{pq},\gamma_{qr}, \gamma_{rp}$ of a geodesic triangle $\Delta pqr$ in $\mfld$,} then
\begin{align*}
    a^2 & \leq \zeta(\kappa_{\min}, c) b^2 + c^2 - 2\revisenew{\langle \Log_rq, \Log_rp\rangle} 
\end{align*}
where $\zeta(\kappa,c) := \sqrt{-\kappa}c\coth(\sqrt{-\kappa}c)$.
\end{lemma}
\revise{The second inequality characterizes the length when sectional curvature is bounded from above. In particular, if the upper-bound $\kappa$ is positive, the diameter of the manifold should be bounded for the inequality to hold. In consistency with \cref{asmp: sec-curvature},  we define the upper bound of diameter $D(\kappa)$ as 
\begin{align*}
    D(\kappa) := \begin{cases} \infty , \quad\quad\ \ \kappa \leq 0,\\ \nfrac{\pi}{\sqrt{\kappa}}, \quad\ \kappa > 0. \end{cases}
\end{align*}
Then we can lower bound the length of the sides as follows.}
\begin{lemma} [Corollary 2.1 in \cite{alimisis2020continuous}]  \label{lemma: tri-ineq-nonpos} 
Let $\mfld$ be a Riemannian manifold with sectional curvature bounded above by $\kappa_{\max}$ and diameter $\diam(\mfld) < D(\kappa_{\max})$. \revise{If $a,b,c$ are the length of sides $\gamma_{pq},\gamma_{qr}, \gamma_{rp}$ of a geodesic triangle $\Delta pqr$ in $\mfld$,} then
\begin{align*}
    a^2 & \geq \xi(\kappa_{\max}, c) b^2 + c^2 - 2\revisenew{\langle \Log_rq, \Log_rp\rangle}
\end{align*}
where 
$
     \xi(\kappa, c) := \begin{cases} \sqrt{-\kappa}c\coth(\sqrt{-\kappa}c), \quad\ \ \ \kappa \leq 0,\\  \sqrt{\kappa}c\cot(\sqrt{\kappa}c), \qquad\quad\ \ \ \ \kappa > 0. \end{cases}
$
\end{lemma}
\begin{remark}
When $\kappa$ is set to $0$, both $\zeta(\kappa,\cdot)$ and $\xi(\kappa, \cdot)$ reduce to $1$. 
\end{remark}

The upper and lower bounds in the above lemmas decide the minimal and maximal distortion rates. We define a ratio $\tau$ between the two rates to quantify how curvature changes in the space as: 
\begin{equation*}\label{eq:tau}
    \tau([\kappa_{\min}, \kappa_{\max}], c) := \begin{cases} \sqrt{|\tfrac{\kappa_{\min}}{\kappa_{\max}}|}\cdot \nfrac{\coth\left(\sqrt{|\kappa_{\min}|}c\right)}{\coth\left(\sqrt{|\kappa_{\max}|}c\right)}, \quad\  \kappa_{\max} \leq 0, \\ \sqrt{|\tfrac{\kappa_{\min}}{\kappa_{\max}}|} \cdot \nfrac{\coth\left(\sqrt{|\kappa_{\min}|}c\right)}{\cot\left(\sqrt{\kappa_{\max}}c\right)}, \quad\quad\quad\kappa_{\max} > 0. \end{cases}
\end{equation*}

We emphasize $\tau$ is defined as the ratio between maximal distortion $\zeta$ and $\xi$, i.e.,  
\begin{equation} \label{eq: tau-ratio}
    \tau([\kappa_{\min}, \kappa_{\max}], D) =  \zeta(\kappa_{\min}, D) / \xi(\kappa_{\max}, D).
\end{equation}

\subsection{Riemannian corrected extragradient} \label{sec: single-loop}
We present a Riemannian extragradient method with an additional correction term (RCEG) for geodesically convex-concave $f$ (see \cref{algo: extragrad}). We overload manifold operations to have more compact notation for the Riemannian gradient step of pair $(x, y) \in \mfld \times \mfldn$:
\begin{equation}
    \Exp_{(x,y)}(u, v) := (\Exp_x(u),\ \Exp_y(v)).
\end{equation}
We use a geodesic averaging scheme \cite{tripuraneni2018averaging,zhang2016first} in \cref{algo: extragrad}: i.e., at each iteration we calculate
\begin{equation} \label{eq: geodesic-average}
    (\overline{w}_{t+1}, \overline{z}_{t+1}) = \Exp_{(\overline{w}_{t}, \overline{z}_{t})}\left(\frac{1}{t+1} \cdot \Log_{\overline{w}_t}(w_{t+1}),\ \frac{1}{t+1} \cdot \Log_{\overline{z}_t}(z_{t+1})\right).
\end{equation}
\revise{This averaging implies that at iteration $t$, the point $\bar{w}_t$ lies on the geodesic from $\bar{w}_{t-1}$ to $w_t$, and $\bar{z}_t$ lies on the geodesic from $\bar{z}_{t-1}$ to $z_t$.} The output produced by averaging is then $(\overline{w}_T, \overline{z}_T)$. The following theorem shows that the averaged output of RCEG achieves a curvature-dependent convergence rate for smooth convex-concave $f$ on Riemannian manifolds.
\begin{theorem} \label{thm: convex-concave} 
Suppose \crefrange{asmp: smooth}{asmp: sec-curvature} hold, and the iterations remain in  subdomains\footnote{The condition allows an upper-bound for distortion (\emph{cf.}~$\tau$) and is regular in Riemannian optimization literature \cite{alimisis2020continuous,zhang2016first}.} of bounded diameter $D_\mfld$ and $D_\mfldn$. Let $(x_t,$ $y_t,w_t,z_t)$ be the sequence obtained from the iteration of \cref{algo: extragrad} with initialization $x_1 = w_1$, $y_1 = z_1$. Then, using a \revise{step-size} $\eta = \tfrac{1}{2L\sqrt{\tau}}$, the following inequality holds for $T$:
\begin{align*}
    \max_{y \in \mfldn} f(\overline{w}_T, y) - \min_{x \in \mfld} f(x, \overline{z}_T)  \leq \frac{d^2_{\mfld}(x_1, x^*) + d_{\mfldn}^2(y_1, y^*)}{\eta T},
\end{align*}
with $(\overline{w}_T, \overline{z}_T)$ obtained via averaging in~\cref{eq: geodesic-average}, and $\tau = \tau([\kappa_{\min}, \kappa_{\max}],\max(D_\mfld, D_\mfldn))$. 
\end{theorem}
\cref{thm: convex-concave} is a natural nonlinear extension of the known result achieved by extragradient method in the Euclidean setting. 
\revise{We notice that, different from Riemannian minimization algorithms (e.g., \cite{zhang2016first}), 
whenever the lower and upper of curvature coincide, the curvature-free convergence rate can be retrieved.}

\begin{algorithm}[t]
\caption{Riemannian Corrected Extragradient (RCEG)}
\label{algo: extragrad}
\begin{algorithmic}[1]
\STATE{\textbf{input:} objective $f$, initialization $(x_1, y_1)$, step-size $\eta$}
\STATE{Set $w_1 \leftarrow x_1$, $z_1 \leftarrow y_1$}
\FOR{$t = 1, 2, \dots, T $}
\STATE{$(w_{t}, z_{t}) \leftarrow \Exp_{(x_t, y_t)}(- \eta \nabla_x f(x_t,y_t), \eta \nabla_y f(x_t,y_t))$}
\STATE{$(x_{t+1}, y_{t+1}) \leftarrow \Exp_{(w_{t}, z_{t})}(- \eta \nabla_x f(w_{t},z_{t}) + \Log_{w_{t}}(x_t), \eta \nabla_y f(w_{t},z_{t}) + \Log_{z_{t}}(y_t))$}
\ENDFOR
\STATE{\textbf{output:} geodesic averaging scheme $(\overline{w}_T,\overline{z}_T)$ as in \cref{eq: geodesic-average}}
\end{algorithmic}
\end{algorithm}

\textbf{\emph{The correction term in RCEG.}} The translation of the extragradient method to Riemannian manifolds is non-trivial. We briefly elaborate on the proof technique and focus on the update of $x_t$ for simplicity. For any $x\in \mfld$, at each step, we need to bound the difference as
\begin{equation} \label{eq: step-bound}
    f(w_t, z_t) - f(x, z_t) \leq a_{t} \cdot d_{\mfld}^2(x_{t+1}, x) - b_t \cdot d_{\mfld}^2(x_t, x),
\end{equation}
where $a_t$,$b_t$ are undetermined constants. We start with geodesic convexity, i.e., $f(w_t, z_t) - f(x, z_t) \leq  -\langle \nabla_x f(w_t, z_t), \Log_{w_t}(x)\rangle$.
The correction term in RCEG allows a useful equality 
\begin{equation} \label{eq: rceg-decomposition}
    \Log_{w_t}(x_{t+1}) = \Log_{w_t}(x_{t}) - \eta \nabla_x f(w_t, z_t). 
\end{equation}
This equality leads to a decomposition of cross terms in $\langle \nabla_x f(w_t, z_t), \Log_{w_t}(x)\rangle$, and as a result, we obtain
\begin{align} \label{eq: rceg-bound}
    f(w_t, z_t) - f(x, z_t) \leq \tfrac{1}{\eta} \langle \Log_{w_t}(x_{t+1}), \Log_{w_t}(x) \rangle - \tfrac{1}{\eta} \langle \Log_{w_t}(x_{t}), \Log_{w_t}(x) \rangle.
\end{align}
Applying comparison inequalities on \eqref{eq: rceg-bound} leads to an efficient upper-bound in \eqref{eq: step-bound}. By telescoping on \eqref{eq: step-bound}, we obtain the convergence result. 

It is worth noting that the correction term is crucial to our Riemannian convergence analysis. In the Euclidean case, the extragradient update is simply realized as $x_{t+1} \leftarrow x_{t} - \eta \nabla_x f(w_{t}, z_{t})$. However, \revisenew{we cannot prove using the current technique that, a direct Riemannian counterpart, i.e., $x_{t+1} \leftarrow \Exp_{ x_{t}}( - \eta \Gamma^{x_t}_{w_{t}} \nabla_x f(w_{t}, z_{t}) )$, is a convergent algorithm. This is due to it does not permit a decomposition as in \cref{eq: rceg-decomposition}, and necessitates bounding the cross-term $\langle \Gamma_{w_t}^{x_t}\nabla_x f(w_t, z_t), \Log_{x_t}(x) - \Gamma_{w_t}^{x_t}\Log_{w_t}(x) \rangle$. This approach leads to error terms caused by non-linear geometry that we cannot upper-bound.}

\subsection{Lemmas for proving Theorem~\ref{thm: convex-concave}} \label{sec: lemma-algo}
Before proceeding to the main proof, we first present a lemma that characterizes the behavior of the geodesic averaging scheme in \cref{eq: geodesic-average} under convex-concave setting.
\begin{lemma} \label{lemma: averaging}
Suppose \cref{asmp: convex-concave} holds. Then, for any iterates $(w_t, z_t)$, the geodesic averaging scheme $(\overline{w}_t, \overline{z}_t)$ as in \cref{eq: geodesic-average} satisfies, for any positive integer $T$ and any $x\in X$, $y\in Y$ the following bound:
\begin{align*}
    f(\overline{w}_T, y) - f(x, \overline{z}_T) \leq \frac{1}{T} \cdot \sum_{t=1}^T [f(w_t, y) - f(x, z_t)].
\end{align*}
\end{lemma}
\begin{proof}
  For the case $T = 1$, the result trivially holds. Now suppose the condition holds for $T - 1$; \revise{then we have already 
  \begin{align}\label{eq: average-induction}
    f(\overline{w}_{T-1}, y) - f(x, \overline{z}_{T-1}) \leq \frac{1}{T-1} \cdot \sum_{t=1}^{T-1} [f(w_t, y) - f(x, z_t)].  
    \end{align}}%
    \revisenew{
    We want to show $\overline{w}_{T}$ and $\overline{z}_{T}$ lie, respectively, on the geodesics connecting $\overline{w}_{T-1}, w_{T}$ and $\overline{z}_{T-1}, z_{T}$. For any $x\in \mfld$ and $u\in T_x\mfld$, let us use the temporary notation $\gamma(t;u)$ to denote the geodesic $\gamma$ with $\gamma(0) = x$ and $\gamma'(0) = u$. From Lemma 5.18 in \cite{lee2006riemannian}, it holds that $\gamma(s;cu) = \gamma(cs;u)$ for any $c,s\in [0,1]$. Now we denote $u_t = \Log_{\bar{w}_t}(\bar{w}_{t+1})$ and $\alpha_t = 1/(t+1)$. By the definition of $\Log$, $\Exp$ map and averaging scheme in \eqref{eq: geodesic-average} we have 
    \begin{align*}
        \bar{w}_{t+1} = \Exp_{\bar{w}_t} (\alpha_t u_t) = \gamma(1;\alpha_t u_t) = \gamma(\alpha_t;u_t).
    \end{align*}
    where $\gamma(\cdot,u_t)$ is the geodesic connecting $\bar{w}_t$ and $w_{t+1}$. So $\bar{w}_{t+1}$ lies on the geodesic from $\bar{w}_t$ to $w_{t+1}$. In a similar way we can also show that $\bar{z}_{t+1}$ lies on the geodesic from $\bar{z}_t$ to $z_{t+1}$.}
    Then we can calculate 
    \begin{align*}
    f(\overline{w}_{T}, y) &- f(x, \overline{z}_{T}) \\
    &\leq\quad\frac{T-1}{T} \cdot [f(\overline{w}_{T-1}, y) - f(x, \overline{z}_{T-1})] + \frac{1}{T} \cdot [f(w_{T}, y) - f(x, z_{T})] \\
    &\leq\quad\frac{T-1}{T} \cdot \frac{1}{T-1} \sum_{t=1}^{T-1} [f(w_t, y) - f(x, z_t)] + \frac{1}{T} \cdot [f(w_{T}, y) - f(x, z_T)] \\
    &=\quad\frac{1}{T} \sum_{t=1}^T [f(w_t, y) - f(x, z_t)],
  \end{align*}
  where the first inequality comes from the fact the $f$ is geodesically convex-concave and the second is due to induction in \cref{eq: average-induction}.
\end{proof}

\revise{The next several technical lemmas help bound the iterates of \cref{algo: extragrad}. We state them as a preparation to the proof of \cref{thm: convex-concave}. The first prepares an inequality where we perform a telescopic sum on the distance to the saddle point.
\begin{lemma} \label{lemma: rceg-1}
Suppose the same condition in \cref{thm: convex-concave}. Then for the iterates $(x_t,y_t,w_t,z_t)$ produced by \cref{algo: extragrad}, it holds that
    \begin{align*} 
    f(w_t, y) - f(x, z_t) 
    \leq & \tfrac{1}{\eta} \left( d^2_{\mfld}(x,x_{t})  + d^2_{\mfldn}(y,y_{t}) - d^2_{\mfldn}(y,y_{t+1}) - d^2_{\mfld}(x,x_{t+1})\right)\\
    - \tfrac{\xi}{\eta}  d^2_{\mfld}(w_t,x_t) & - \tfrac{\xi}{\eta} d^2_{\mfldn}(z_t,y_t) +  \tfrac{\zeta}{\eta} d^2_{\mfld}(w_t,x_{t+1}) + \tfrac{\zeta}{\eta} d^2_{\mfldn}(z_t,y_{t+1}),
    \end{align*}
    \revisenew{where $\xi = \xi(\kappa_{\max}, \max(D_\mfld, D_\mfldn))$ and $ \zeta = \zeta(\kappa_{\min}, \max(D_\mfld, D_\mfldn))$}.
\end{lemma}
\begin{proof}
Since $f$ is geodesically convex in $x$ and geodesically concave in $y$, for any two points $x \in X, y \in Y$, the following inequality holds
\begin{align} \label{eq: ineq-1}
    f(w_t, y) - f(x, z_t) & =  f(w_t, z_t) -  f(x, z_t) - (f(w_t, z_t) - f(w_t, y))  \nonumber\\
    & \leq - \langle \nabla_x f(w_t,z_t), \Log_{w_t}(x) \rangle + \langle \nabla_y f(w_t,z_t), \Log_{z_t}(y)\rangle.
\end{align}
Recalling the iteration of RCEG in \cref{algo: extragrad}:
\begin{align*}
    x_{t+1} \leftarrow &  \Exp_{w_t}(-\eta\nabla_x f(w_t,z_t) + \Log_{w_t}(x_t)), \\
    y_{t+1} \leftarrow &  \Exp_{z_t}(\eta\nabla_y f(w_t,z_t) + \Log_{z_t}(y_t)),
\end{align*}
by the definition of inverse exponential map, we have
\begin{align*}
    \eta \nabla_x f(w_t,z_t) &= \Log_{w_t}(x_t) - \Log_{w_t}(x_{t+1}), \\
    - \eta \nabla_y f(w_t,z_t) &= \Log_{z_t}(y_t) - \Log_{z_t}(y_{t+1}).
\end{align*}
This allows us to decompose these mixed terms in the right-hand side of \cref{eq: ineq-1} as
\begin{align*}
    \eta \langle \nabla_x f(w_t,z_t), \Log_{w_t}(x) \rangle & = \langle \eta \nabla_x f(w_{t},z_{t}) - \Log_{w_{t}}(x_t) , \Log_{w_{t}}(x) \rangle \\
    & + \langle \Log_{w_{t}}(x_t), \Log_{w_{t}}(x) \rangle \\
    & = - \langle \Log_{w_{t}}(x_{t+1}), \Log_{w_{t}}(x) \rangle + \langle \Log_{w_{t}}(x_t), \Log_{w_{t}}(x) \rangle, \\
    \eta \langle \nabla_y f(w_t,z_t), \Log_{z_t}(y) \rangle & = \langle \eta \nabla_y f(w_{t},z_{t}) + \Log_{z_{t}}(y_{t}), \Log_{z_{t}}(y) \rangle \\
    & -  \langle \Log_{z_{t}}(y_t), \Log_{z_{t}}(y) \rangle \\
    & = + \langle \Log_{z_{t}}(y_{t+1}), \Log_{z_{t}}(y) \rangle - \langle \Log_{z_{t}}(y_t), \Log_{z_{t}}(y) \rangle.
\end{align*}
Plugging this decomposition back into \cref{eq: ineq-1} results in following inequality: 
\begin{align} \label{eq: ineq-2}
\begin{split}
     f(w_{t}, y) - f(x, z_{t}) & \leq  \tfrac{1}{\eta} \langle \Log_{w_{t}}(x_{t+1}), \Log_{w_{t}}(x) \rangle - \tfrac{1}{\eta} \langle \Log_{w_{t}}(x_t), \Log_{w_{t}}(x) \rangle \\
    & + \tfrac{1}{\eta} \langle \Log_{z_{t}}(y_{t+1}), \Log_{z_{t}}(y) \rangle - \tfrac{1}{\eta} \langle \Log_{z_{t}}(y_t), \Log_{z_{t}}y \rangle.
\end{split}
\end{align}
Now, it suffices to bound the right-hand side of \cref{eq: ineq-2} by leveraging comparison inequalities on Riemannian manifolds with bounded sectional curvature. \revisenew{Combining the bounded domain condition and Lemma~\ref{lemma: tri-ineq-nonpos}, we then obtain}
\begin{align*}
    - 2\langle \Log_{w_{t}}(x_{t}), \Log_{w_{t}}(x) \rangle & \leq - \xi d_{\mfld}^2(w_{t}, x_t) - d_{\mfld}^2(w_{t}, x) + d_{\mfld}^2(x_{t}, x) \\
    -2\langle \Log_{z_{t}}(y_{t}), \Log_{z_{t}}(y) \rangle & \leq - \xi d_{\mfldn}^2(z_{t}, y_t) - d_{\mfldn}^2(z_{t}, y) + d_{\mfldn}^2(y_{t}, y).
\end{align*}
\revisenew{Similarly, we use Lemma~\ref{lemma: tri-ineq-lower} and bounded domain to obtain}
\begin{align*}
    2\langle \Log_{w_{t}}(x_{t+1}), \Log_{w_{t}}(x) \rangle 
    \leq & \ \zeta d_{\mfld}^2(w_t, x_{t+1}) + d_{\mfld}^2(w_{t}, x) - d_{\mfld}^2(x_{t+1}, x), \\
    2\langle \Log_{z_{t}}(y_{t+1}), \Log_{z_{t}}(y) \rangle 
    \leq &  \ \zeta d_{\mfldn}^2(z_t, y_{t+1}) + d^2_{\mfldn}(z_{t}, y) - d^2_{\mfldn}(y_{t+1}, y).
\end{align*}
Inserting the above inequalities to \cref{eq: ineq-2} yields the desired inequality. 
\end{proof}
Before proceeding, we need the following lemma.
\revisenew{\begin{lemma} \label{lemma: geodesic-transport}
    For any two points $x,y\in\mfld$, it holds that $\Log_yx = -\Gamma_x^y\Log_xy$.
\end{lemma}
\begin{proof}
    Suppose $\gamma_x$ is the geodesic between $x$ and $y$, i.e. $\gamma_x(0)= x$ and $\gamma_x(1) = y$. Hence, $\gamma_y(t) = \gamma_x(1-t)$ is also a geodesic with $\gamma_y(0)=y$ and $\gamma_y(1)=x$. By the chain rule and the definition of exponential map the following holds:
    \begin{align*}
        \Log_yx = \gamma'_y(0) = - \gamma'_x(1).
    \end{align*}
    We consider the parallel transport of $\Log_x^y$ along geodesic $\gamma_x$. Then by \cite{lee2006riemannian}, there exists a unique vector field $v(t)$ along $\gamma_x$ such that 
    \begin{align}  \label{eq: pt-geodesic}                       
        \nabla_{\gamma'_x(t)}v(t) =  0 \qquad \text{and} \qquad v(0) = \Log_x^y.
    \end{align}
    We notice that $v(t) = \gamma'_x(t)$, the following condition holds due to the geodesic equation
    \begin{align*}             
        \nabla_{\gamma'_x(t)}\gamma'_x(t) =  0.
    \end{align*}
    Also, we have $\gamma'_x(0) = \Log_xy$. Then $\gamma_x'(t)$ is the unique vector field satisfying \eqref{eq: pt-geodesic} and hence we conclude
    \begin{align*}
        \Gamma_x^y\Log_xy = v(1) = \gamma'_x(1) = -\Log_yx.
    \end{align*}
\end{proof}
}
The next lemma states that the error terms scale quadratically in the step-size $\eta$.
\begin{lemma} \label{lemma: rceg-2}
Suppose the same condition in \cref{thm: convex-concave}. Then for the iteration $(x_t,y_t,w_t,z_t)$ produced by \cref{algo: extragrad}, it holds that
\begin{align*}
    d^2_{\mfld}(w_t,x_{t+1}) & \leq \eta^2L^2 \cdot (d^2_{\mfld}(w_t,x_{t}) + d^2_{\mfldn}(y_t,z_{t})), \\
    d^2_{\mfldn}(z_t,y_{t+1}) & \leq \eta^2L^2 \cdot (d^2_{\mfld}(w_t,x_{t}) + d^2_{\mfldn}(y_t,z_{t})).
\end{align*}
\end{lemma}
\begin{proof}
We first recall from the iteration of \cref{algo: extragrad} and definition of inverse exponential map that
\begin{align*}
    \Log_{w_t}(x_{t+1}) & = \Log_{w_t}(x_t) - \eta\nabla_x f(w_t,z_t), \\
    \Log_{z_t}(y_{t+1}) & = \Log_{z_t}(y_t) + \eta \nabla_y f(w_t,z_t).
\end{align*}
Using the definition of Riemannian distance, we have
\begin{align} \label{eq: ineq-3}
\begin{split}
    d_{\mfld}^2(w_t, x_{t+1}) & = \| \Log_{w_t}(x_{t+1})\|^2 = \| \eta \nabla_x f(w_{t}, z_{t}) - \Log_{w_{t}}(x_{t}) \|^2, \\
    d_{\mfldn}^2(z_t, y_{t+1}) & = \| \Log_{z_t}(y_{t+1})\|^2 = \| \eta \nabla_y f(w_{t},z_{t}) + \Log_{z_t}(y_{t}) \|^2.
\end{split}
\end{align}
Next, we utilize Lemma~\cref{lemma: geodesic-transport} and obtain
\begin{align*}
    \Log_{w_t}(x_t) = - \Gamma_{x_t}^{w_t}\Log_{x_t}(w_t) = \eta\Gamma_{x_t}^{w_t}\nabla_x f(x_t,y_t), \\
    \Log_{z_t}(y_t) = - \Gamma_{y_t}^{z_t}\Log_{y_t}(z_t) = -\eta\Gamma_{y_t}^{z_t}\nabla_y f(x_t,y_t).
\end{align*}
Plugging the above equalities into \cref{eq: ineq-3} yields our result
\begin{align*}
    &d_{\mfld}^2(w_t, x_{t+1})  = \eta^2 \| \nabla_x f(w_{t}, z_{t}) - \Gamma_{x_t}^{w_{t}}\nabla_x f(x_t, y_t) \|^2 \leq \eta^2 L^2 \left( d^2_{\mfld}(x_t, w_{t}) + d^2_{\mfldn}(y_t, z_{t})\right) 
    \end{align*}
and
\begin{align*}
    &d_{\mfldn}^2(z_t, y_{t+1})  = \eta^2 \| \nabla_y f(w_{t},z_{t}) - \Gamma_{y_t}^{z_{t}}\nabla_y f(x_t, y_t) \|^2 \leq \eta^2 L^2 \left( d^2_{\mfld}(x_t, w_{t}) + d^2_{\mfldn}(y_t, z_{t})\right)
\end{align*}
where the inequalities are due to $L$-smoothness.
\end{proof}

\subsection{Proof of Theorem~\ref{thm: convex-concave}} \label{sec: proof-algo}
Finally, with these building blocks, we can present the formal proof of \cref{thm: convex-concave}.
\begin{proof}[Proof of \cref{thm: convex-concave}]
Starting with~\cref{lemma: rceg-1}, we immediately have for any $x\in X$, $y\in Y$,
   \begin{align*} 
    f(w_t, y) - f(x, z_t) 
    \leq & \tfrac{1}{\eta} \left( d^2_{\mfld}(x,x_{t})  + d^2_{\mfldn}(y,y_{t}) - d^2_{\mfldn}(y,y_{t+1}) - d^2_{\mfld}(x,x_{t+1})\right)\\
    - \tfrac{\xi}{\eta}  d^2_{\mfld}(w_t,x_t) & - \tfrac{\xi}{\eta} d^2_{\mfldn}(z_t,y_t) +  \tfrac{\zeta}{\eta} d^2_{\mfld}(w_t,x_{t+1}) + \tfrac{\zeta}{\eta} d^2_{\mfldn}(z_t,y_{t+1}).
    \end{align*}
\revisenew{We can bound the last two terms, i.e. $ d^2_{\mfld}(w_t,x_{t+1})$, $ d^2_{\mfldn}(z_t,y_{t+1})$, with \cref{lemma: rceg-2}:
\begin{align*}
    d^2_{\mfld}(w_t,x_{t+1}) + d^2_{\mfldn}(z_t,y_{t+1}) \leq 2\eta^2L^2 \cdot (d^2_{\mfld}(w_t,x_{t}) + d^2_{\mfldn}(y_t,z_{t})).
\end{align*}
Based on our parameter choice $\eta = \frac{1}{2L\sqrt{\tau}}$ where $\tau = \zeta / \xi$ (\emph{c.f.}, \Cref{eq: tau-ratio}), it holds that
\begin{align*}
    - \xi  d^2_{\mfld}(w_t,x_t) & - \xi d^2_{\mfldn}(z_t,y_t)  +  \zeta d^2_{\mfld}(w_t,x_{t+1}) + \zeta d^2_{\mfldn}(z_t,y_{t+1}) \\
    & \leq -(\xi - 2\eta^2L^2 \zeta) \cdot d^2_{\mfld}(w_t,x_t) -  (\xi - 2\eta^2L^2 \zeta) \cdot d^2_{\mfldn}(z_t,y_t) \leq 0.
\end{align*}
We plug this into the above inequality and obtain
}
  \begin{align*} 
    f(w_t, y) - f(x, z_t) 
    \leq  \tfrac{1}{\eta} \left( d^2_{\mfld}(x,x_{t})  + d^2_{\mfldn}(y,y_{t}) - d^2_{\mfldn}(y,y_{t+1}) - d^2_{\mfld}(x,x_{t+1})\right).
    \end{align*}
Summing from $1$ to $T$, we further obtain
\begin{align*}
    \sum_{t=1}^T f(w_t, y) - f(x, z_t)  \leq \frac{1}{\eta} \cdot \left( d_{\mfld}^2(x_1, x) + d_{\mfldn}^2(y_1, y) \right).
\end{align*}
Lastly, by Lemma~\ref{lemma: averaging}, the averaging scheme satisfies
\begin{align*}
    f(\overline{w}_T, y^*) -  f(x^*, \overline{z}_T)  \leq \frac{1}{T} \sum_{t=1}^T \left(f(w_t, y) - f(x, z_t)\right) \leq \frac{d_{\mfld}^2(x_1, x^*) + d_{\mfldn}^2(y_1, y^*)}{\eta T},
\end{align*}
where $(x^*,y^*)$ is the global saddle point pair. Hence the result follows.
\end{proof}
}


%% file: sec_5_applications.tex
\section{Applications and experiments} \label{sec: empirical}
In this section, we confirm the theoretical and algorithmic results of our work through two experiments.
\subsection{Strong duality for constrained Hadamard optimization}
From the theoretical aspect, we show the prowess of \cref{thm: geodesic-sion} and \cref{cor: riemannian-sion} by establishing a strong-duality result for the constrained Hadamard optimization problem. The setting is previously entailed in \Cref{sec: examples}. 
\begin{corollary}\label{cor: lagrangian}
Consider the constrained optimization problem in \cref{eq: constrained-hadamard} on a Hadamard manifold $\mfld$. \revise{If $X \subseteq \mfld$ is compact and geodesically convex,}  both $g$ and each $h_i$, $i = 1, \dots n$, are lower semi-continuous and geodesically convex, then the Lagrangian $f(x, \lambda) = g(x) + \langle h(x), \lambda \rangle$ satisfies the following identity:
\begin{align*}
    \revise{\min_{x \in X} \sup_{\lambda \in \bbR^n} f(x, \lambda) = \sup_{\lambda \in \bbR^n} \min_{x\in X} f(x, \lambda).}
\end{align*}
\end{corollary}
\begin{proof}
This claim is immediate from \cref{thm: geodesic-sion}.
\end{proof}
Similar to the Euclidean case, \cref{cor: lagrangian} guarantees that the minimizer for \cref{eq: constrained-hadamard} can be efficiently found by maximizing the dual problem $g_{\alpha}^*(\lambda) = \min_x f_\alpha(x,\lambda)$. We point out a similar result can be found in \cite{yang2014riem}, which establishes a KKT theorem for constrained Hadamard optimization problem.

\subsection{Robust SPD-PCA}
We demonstrate the tractability of our RCEG by conducting numerical experiments on solving robust SPD-PCA. While the problem in \cref{eq: robust-pca} is difficult in the Euclidean space, we show that it can be efficiently solved under our geodesic convex-concave setting. 

Before we present the experiments, we comment on how well this problem could align with our assumptions. More precisely, the SPD manifold is Hadamard with its curvature in range $[-\tfrac{1}{2}, 0]$ \revise{\cite{criscitiello2021negative,dolcetti2018differential}}. The sphere manifold is of positive curvature $+1$ and is a complete manifold. 
\revise{From the definition of Riemannian distance, $f_{\alpha}(x,M)$ in~\eqref{eq: robust-pca} is geodesic strongly-concave and smooth in $M$ (\cite{alimisis2020continuous}). \revisenew{It is also trivial to verify that $f_{\alpha}$ is geodesic smooth and locally convex} in $x$ around the top eigenvector $u_1$ of $M$, which is also the minimizer given $M$. By considering a small level set around $u_1$, we can apply our \cref{thm: geodesic-sion} to guarantee the existence of saddle point.}

\revisenew{We compare our RCEG with the Riemannian gradient descent-ascent (RGDA) method.  The RDGA method performs the following iteration:
\begin{align*}
    (x_{t+1}, y_{t+1}) = \Exp_{(x_t,y_t)}(-\eta \nabla_x f(x_t,y_t), \eta \nabla_y f(x_t,y_t)).
\end{align*}
We also consider the effect of averaging scheme in \cref{eq: geodesic-average}. We will use RCEG-last and RCEG-ave to denote, respectively, the last-iterate version and the average version of RCEG. Respectively, RGDA-last and RGDA-ave refer to the cases when the algorithm outputs the last-iterate and the average average iterate.
}

\paragraph{Data generation} We run our test on a synthetic dataset $\{M_i \in \spd(n) \}_{i=1}^k$. For each $M_i$, we first produce $B \in \bbR^{n \times n}$ with i.i.d. random entries from standard Gaussian distribution and compute its QR decomposition as $B = QR$, where $Q\in \bbR^{n \times n}$ is orthonormal matrix and $R \in \bbR^{n \times n}$ is upper-triangular matrix. We then generate random eigenvalue $\sigma = (\sigma_1, \dots, \sigma_n)$ in range $[\mu, L]$. Finally we obtain $M_i = Q\diag(\sigma)Q^\top$.

\paragraph{Results} The empirical performance of Riemannian minimax algorithms is illustrated in \cref{fig: robust-pca}. \revisenew{The RCEG-last is able to converge in an almost linear rate in {\it later stages}, whereas the RCEG-ave converges in a slow sublinear rate. The difference is due to the gradient dominance and local geodesic strong-convexity and strong-concavity of our objective \cite{zhang2016first}: while the sublinear rate of average regime is predicted by \cref{thm: convex-concave}, the fast rate of last-iterate regime can be explained by a recent follow-up paper \cite{jordan2022first} which shows that RCEG-last achieves linear rate for geodesic strongly-convex-strongly-concave objectives. 

}

\begin{figure}[ht]
    \centering
    \includegraphics[width=\textwidth]{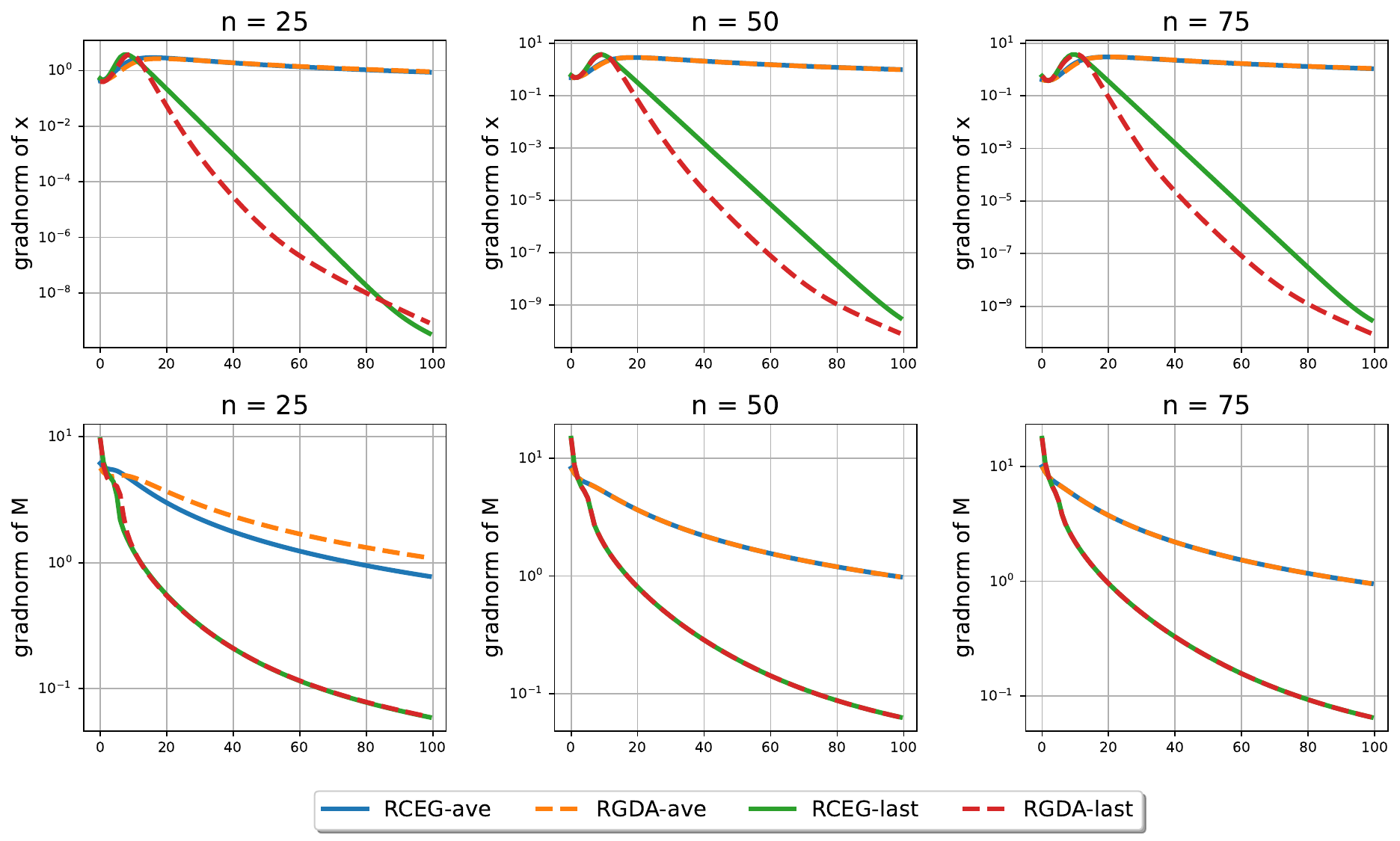}
    \caption{Convergence of RCEG for robust SPD-PCA.  We use step-size $\eta=0.1$, $k=40$, penalty term $\alpha=3$, $\mu = 0.2$ and $L=4.5$ for different trajectories.}
    \label{fig: robust-pca}
\end{figure}

\revise{
\subsection{Robust Karcher mean}
We illustrate the power of our RCEG by a second experiment on the robust Karcher mean problem. Defined over the SPD manifold, \eqref{eq: karcher} is a globally strongly-convex-strongly-concave function for a properly chosen $\alpha$. Therefore we can apply our \cref{thm: geodesic-sion} to guarantee the existence of saddle point.

\paragraph{Setting and results} We run the test on a synthetic dataset $\{M_i\in\spd(n)\}_{i=1}^k$, which is generated in the same way as the experiment for RPCA. The empirical performance of RCEG and RGDA can be found in \cref{fig: karcher}.}

\begin{figure}[ht]
    \centering
    \includegraphics[width=\textwidth]{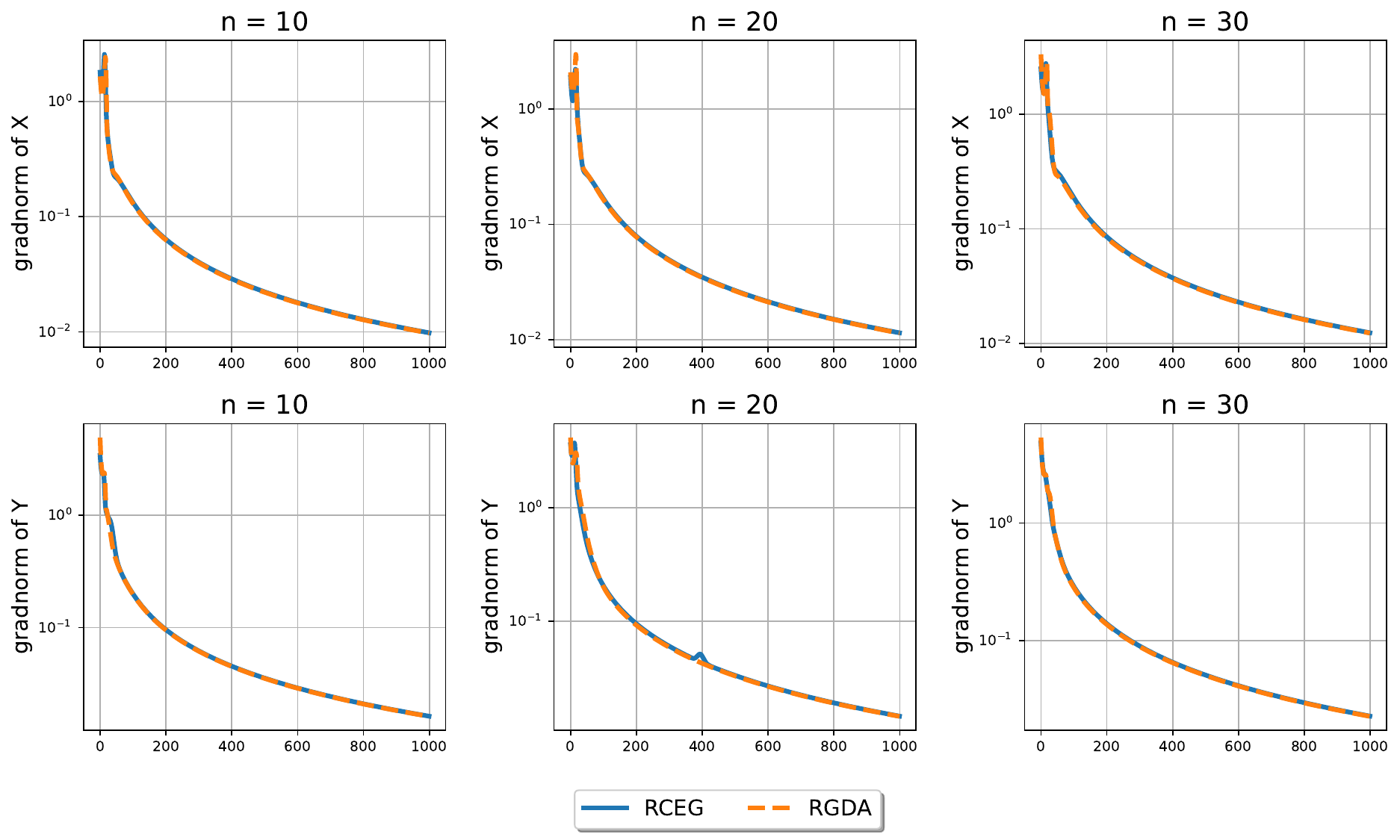}
    \caption{Convergence of RCEG for robust Karcher mean problem. We use step-size $\eta=0.2$, $k=10$, penalty term $\alpha=3$, $\mu = 10^{-3}$ and $L=1$ for different trajectories.}
    \label{fig: karcher}
\end{figure}

\subsection{SPD Bilinear function}
\revise{In this subsection, we provide a synthetic test problem to illustrate the better convergence property of RCEG over RGDA. As the direct generalization of Euclidean gradient-descent method, RGDA is not guaranteed to convergence for convex-concave objectives. We empirically verify this by utilizing $f(x, y) = \langle\Log_{x}(x_0),\Log_y(y_0)\rangle_F$, where $x,y$ belong to the same SPD manifold $\spd(n)$ and $\langle \cdot, \cdot\rangle_F$ is the Frobenius inner product.} Then $f$ formalizes an analogy of Euclidean bilinear function. The result in \cref{fig: bilinear} illustrates that, while our RCEG is convergent, similar to its Euclidean counterpart, the naive Riemannian gradient descent-ascent method can diverge for certain geodesically convex-concave objectives.

\begin{figure}[ht]
    \centering
    \includegraphics[width=1\textwidth]{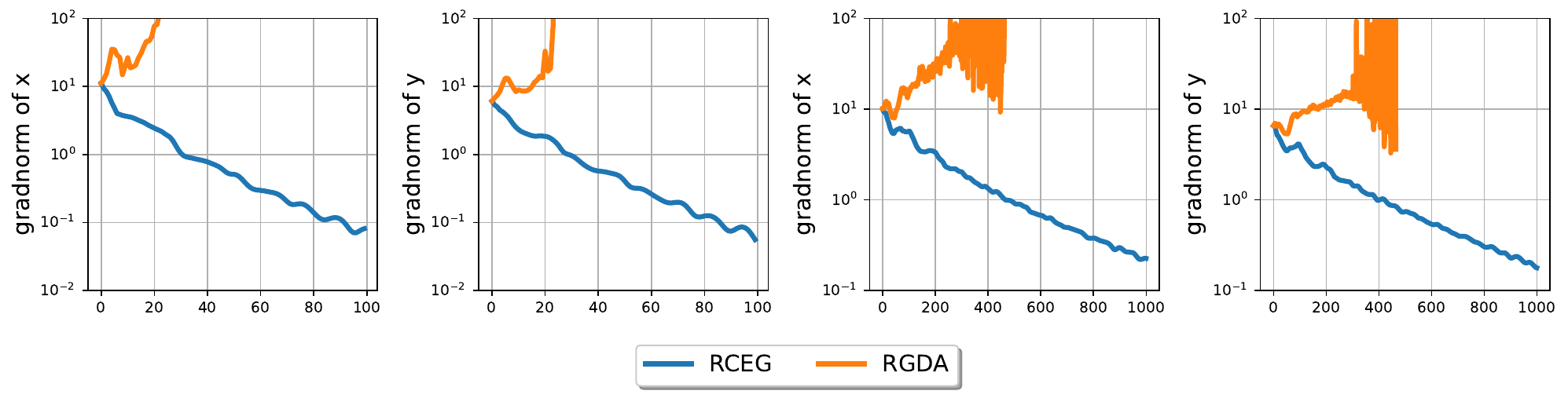}
    \vspace{-0.6em}
    \caption{Comparison between RGDA and RCEG for bilinear objective. While RCEG is convergent, the RGDA method is divergent for minimax problem $f(x,y) = \langle\Log_{x}(x_0),\Log_y(y_0)\rangle_F$, where $x,y$ are defined on $\spd(100)$. We utilize a step-size $\eta = 0.2$.}
    \label{fig: bilinear}
\end{figure}

%% file: sec_67_discussion.tex
\section{Additional related work}
In this section, we will cover some other topics and works relevant to our theme.

\paragraph{Convex-concave minimax algorithms}
The majority of results on minimax optimization leverages the convex-concave setting. The optimal convergence rate for smooth convex-concave problems is $\mathcal{O}(1/\epsilon)$ in terms of duality gap, achieved by mirror-prox method \cite{nemirovski2004prox}, extragradient \cite{mokhtari2020unified} or proximal gradient descent \cite{TSENG1995237}. The rate is matched by the lower-bound analysis in \cite{ouyang2021lower}. Another line \cite{TSENG1995237,mokhtari2020unified,gidel2018variational} studied the strongly-convex-strongly-concave setting, establishing a linear convergence to saddle points. 
Moreover, several works \cite{thekumparampil2019,lin2020near,alkousa2019accelerated} focused on the accelerated algorithms to improve the reliance on conditional number. Specifically, a recent work of \cite{lin2020near} established a near-optimal rate, matching the lower-bound \cite{ouyang2021lower}. 

\paragraph{Nonconvex-nonconcave minimax}
In the general nonconvex-nonconcave minimax problems, determining the existence of global saddle point is NP-hard. Hence a prominent task is to find a well-defined and tractable notion of stationarity. Along this line, works like \cite{jin2020local,mangoubi2021greedy,fiez2021global} investigated different notions of local optimality and their properties. Concurrently, several results \cite{daskalakis2018limit,adolphs2019local,mazumdar2019finding} focused on the relations
between the stable fixed points of algorithm and local stationarity. Another line of research also considers problems with additional structure. For instance, \cite{yang2020global} tackled problem with Polyak-Lojasiewicz (PL) inequality; \cite{diakonikolas2021efficient,malitsky2020golden,liu2019towards} explored Minty variational inequality condition.

\paragraph{Geodesic convex optimization} 
Geodesic convex optimization is a natural extension of convex optimization in Euclidean space onto Riemmanian manifolds. The pioneer work of this field includes \cite{udriste1994convex,absil2009optimization}. More recently, \cite{zhang2016first} provided a first non-asymptotic analysis for Riemannian gradient methods. Subsequent works of the flourishing line explored topics such as acceleration  \cite{zhang2018estimate,pmlr-v125-ahn20a,hamilton2021no,criscitiello2021negative}, variance reduction method \cite{zhang2016riemannian,sato2019riemannian}, and adaptive methods \cite{pmlr-v97-kasai19a}. A parallel line of research tackled constrained Riemannian optimization by studying a hybrid minimax setting, in which $\mfld$ is the Riemannian manifold and $\mfldn$ is Euclidean space. In particular, \cite{khuzani2017stochastic,liu2020simple} formalized the task of constrained geodesic-convex optimization on Riemannian manifold as a minimax problem by augmented Lagrangian method. \cite{huang2020gradient} considered a geodesic-convex-Euclidean-concave minimax problem and analyzed the convergence complexity of a novel Riemannian descent-ascent method. 

\section{Conclusions and Perspectives}
In this work, we provide a new perspective into nonconvex-nonconcave minimax optimization and game theory by considering geodesic convex-concave problems in non-linear geometries. First, we provide an analog of Sion's theorem on geodesic metric spaces. 
Second, we provide novel and efficient minimax algorithm for a different class of geodesic convex-concave games on geodesically complete Riemannian manifolds. We believe our work takes a significant step towards understanding the properties of minimax problems in non-linear geometry, and should help inform the study of many structured learning problems on manifolds. We would like to promote the future investigations and applications by raising several open questions.


\paragraph{Minimax algorithm in metric space}
While most existing literature focuses on Riemannian manifold, few works attempt to tackle the optimization problem in other instances of nonlinear geometry. For example, in~\cite{Bacak+2014,bavcak2013proximal} minimization of convex function in CAT(0) space is considered (geodesic metric spaces of nonpositive curvature), where the notion of a subgradient is absent. In~\cite{bavcak2013proximal} a proximal point algorithm is employed for such settings and shown to admit weak convergence to a minimizer. Since \cref{thm: geodesic-sion} is valid even without the Riemannian metric structure, it lays down a foundation for the study of minimax problems in general metric spaces using proximal operators. 

\paragraph{Acceleration in Riemannian minimax}
Another promising direction is to establish faster rates for Riemannian minimax. Nevertheless, all existing Euclidean accelerated minimax algorithms require accelerated gradient methods as subroutines. Yet, full acceleration without stronger assumptions on curvature and diameter is not possible even for minimization problems, due to \cite{hamilton2021no, criscitiello2021negative}. Nevertheless, partial acceleration is still possible \cite{pmlr-v125-ahn20a}. A potential route for accelerating minimax problems is to consider manifolds with constant curvature \cite{martinez2022global}. However, the result in \cite{martinez2022global} still suffers from an exponential dependence over diameter. We hope these issues can be solved by future works.
 
\paragraph{Matching lower bound} 
Our work opens a pathway to establish upper bound for Riemannian minimax problem. However, a matching lower bound analysis like \cite{ouyang2021lower} is still lacking for minimax problems in Riemannian geometry. 